\documentclass[11pt]{article}

\usepackage[letterpaper,margin=1in]{geometry}
\usepackage{amsmath,amssymb,mathtools}
\usepackage{amsthm}
\usepackage{newtxtext,newtxmath}
\usepackage{bm}
\usepackage{graphicx}
\usepackage{subfig}
\usepackage{algorithm}
\usepackage{algpseudocode}
\usepackage{textcomp}
\usepackage{authblk}
\usepackage{natbib}
\bibpunct[, ]{(}{)}{,}{a}{}{,}

\usepackage[colorlinks=true, linkcolor=blue, citecolor=blue, urlcolor=blue]{hyperref}
\usepackage{microtype}

\newtheorem{theorem}{Theorem}
\newtheorem{lemma}{Lemma}
\newtheorem{corollary}{Corollary}
\theoremstyle{remark}
\newtheorem{remark}{Remark}

\title{Warm-IP: A Path-Following ADMM Warm Start for\\ Interior-Point Quadratic Programming}

\author[a]{Babak Aslani}
\author[b]{Mojtaba Tefagh\thanks{Corresponding author: \texttt{m.tefagh@ed.ac.uk}}}
\author[a]{Gourav Jhanwar}
\author[a]{Masoud Zarepisheh}
\affil[a]{Department of Medical Physics, Memorial Sloan Kettering Cancer Center, New York, USA}
\affil[b]{School of Informatics, The University of Edinburgh, Edinburgh, Scotland}
\date{}

\begin{document}

\maketitle

\begin{abstract}
\noindent
Interior-point (IP) methods are powerful second-order optimization algorithms but can make slow progress during early iterations. We propose Warm-IP, a hybrid framework that combines a path-following variant of the alternating direction method of multipliers (ADMM) with a primal--dual IP method. The ADMM phase generates a warm start that remains close to the central path, enabling the subsequent IP phase to rapidly refine the solution to high accuracy. We evaluate Warm-IP on a diverse collection of real-world and synthetic quadratic programming (QP) instances and compare it with leading commercial solvers (MOSEK and Gurobi) and open-source solvers (Clarabel, PIQP, and OSQP). We also curate and release a new benchmark of large-scale QPs from cancer radiotherapy treatment planning, built from real patient data and provided in canonical QP form, with instances containing hundreds of thousands of constraints. Warm-IP substantially outperforms existing solvers on large-scale problems. On the large-scale radiotherapy benchmark comprising 60 lung cancer cases, for example, Warm-IP solves an instance in 47 seconds on average with no failures, while the strongest competitors, MOSEK and Gurobi, average 55 seconds with failures on 23\% of the cases and roughly 240 seconds with no failures, respectively. To support reproducibility and future benchmarking, we publicly release the implementation and benchmark data through a \href{https://github.com/Radiotherapy-Optimization/Warm-IP}{GitHub repository}. These results demonstrate that appropriately warm-started IP methods can combine the rapid initial progress of first-order methods with the high-accuracy convergence of second-order methods.
\end{abstract}


\section{Introduction}\label{sec:Intro}
	
	Quadratic programming (QP) is a foundational class of optimization problems that arises across a wide range of applications in science, engineering, and operations research. In finance, the Markowitz mean--variance model \citep{markowitz1952portfolio} remains a cornerstone of modern portfolio optimization and is often solved as a QP  \citep{cesarone2023mean,zheng2021quadratic}. In statistical learning, several important models, including the Lasso \citep{tibshirani1996regression} and support vector machines \citep{cortes1995support}, admit QP formulations. In healthcare, large-scale QPs arise naturally in radiotherapy treatment planning, where optimization is used to balance tumor coverage against sparing of healthy organs \citep{zarepisheh2022automated,adibi2022scalable}.
	
	
	\subsection{QP Solution Methods}
	Several solution methods have been developed for QPs, which can generally be categorized into two main classes: first-order and second-order methods.
	
	\textbf{First-order methods}, such as gradient descent \citep{tseng2010coordinate} and the alternating direction method of multipliers (ADMM) \citep{boyd2011distributed}, rely primarily on gradient information. Their iterations involve simple matrix--vector operations with low per-iteration cost and modest memory requirements, which makes them highly scalable for large QP instances; however, their linear or sublinear convergence makes them inefficient when high-precision solutions are needed.
	
	\textbf{Second-order methods}, including interior-point (IP) methods \citep{aps2022mosek,schwan2023piqp,goulart2024clarabel}, quasi-Newton methods \citep{gondzio2019quasi}, and sequential quadratic programming (SQP) methods \citep{jian2022sequential}, exploit curvature information in addition to gradients. Among these, IP methods are particularly attractive for QP: they are robust and reach high accuracy in relatively few iterations, but each iteration requires solving a large Newton (KKT) linear system, and progress is often modest before the fast local convergence regime is reached, which limits scalability on very large instances.
	
	\textbf{Warm-starting IP methods.} A natural strategy for combining the two classes is to warm-start a second-order method from an inexpensive first-order solution. However, warm-starting IP methods is well known to be difficult: an initial point that is close to optimality but poorly centered, typically lying near the boundary of the feasible region, forces the method to take severely damped Newton steps, so a warm-started IP method may perform no better than a cold-started one \citep{yildirim2002warm,gondzio1998warm}. Proposed remedies include carefully modified initial points and controlled perturbations \citep{yildirim2002warm}, exact primal--dual penalty reformulations \citep{benson2007exact}, auxiliary slack variables and relaxations \citep{engau2010interior}, and warm-starting the homogeneous self-dual embedding \citep{skajaa2013warmstarting}. These approaches modify the IP method so that it can tolerate an off-central-path starting point. Warm-IP takes the complementary route: it modifies the first-order phase itself so that its output is centered by construction.

	
	\subsection{Warm-IP: Motivation and Main Idea}
	Figure~\ref{conceptual} illustrates the main motivation behind our approach. First-order methods such as ADMM reduce the optimality gap rapidly during early iterations but slow down when high accuracy is required, while IP methods behave in the opposite way. A natural hybrid strategy is therefore to run ADMM to obtain an approximate solution and then switch to an IP method for high-accuracy refinement. However, a naive ADMM warm start may fail to accelerate the IP method: ADMM iterates are not designed to lie near the central path, so the IP method may spend substantial time re-centering before benefiting from its fast local convergence. Warm-IP instead employs a path-following ADMM whose iterates remain aligned with the central path by construction, producing a warm start that is both close to optimality and compatible with the geometry of IP iterates.
	\begin{figure}[htb]
		\centering
		\subfloat[Comparing ADMM and IP]{%
			\includegraphics[width=0.29\textwidth]{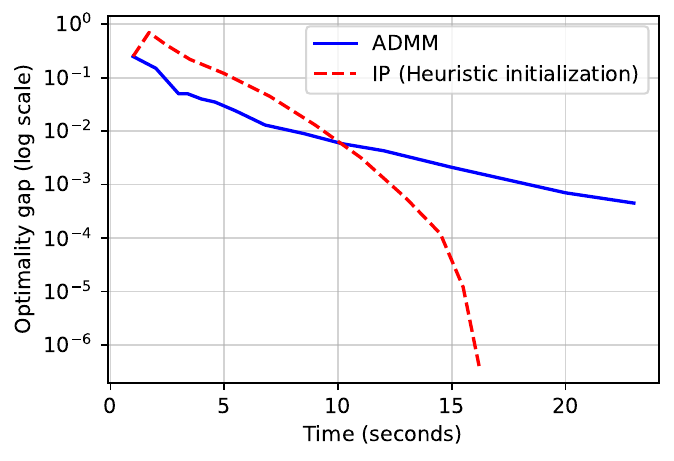}
		}\hfill
		\subfloat[Warm-start IP with ADMM and path-following ADMM]{%
			\includegraphics[width=0.29\textwidth]{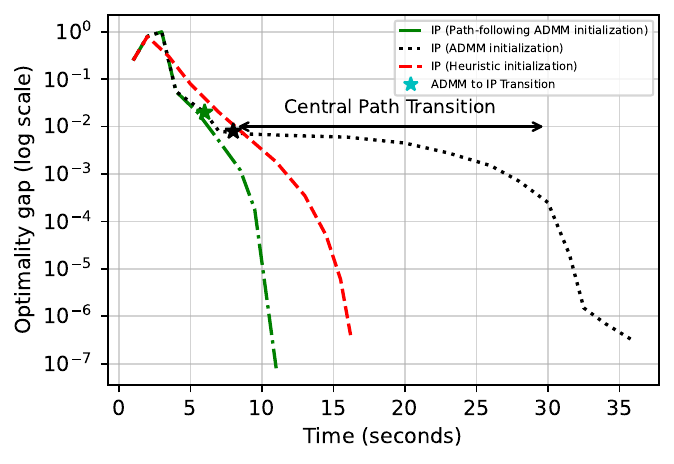}
		}\hfill
		\subfloat[Central Path]{%
			\includegraphics[width=0.27\textwidth]{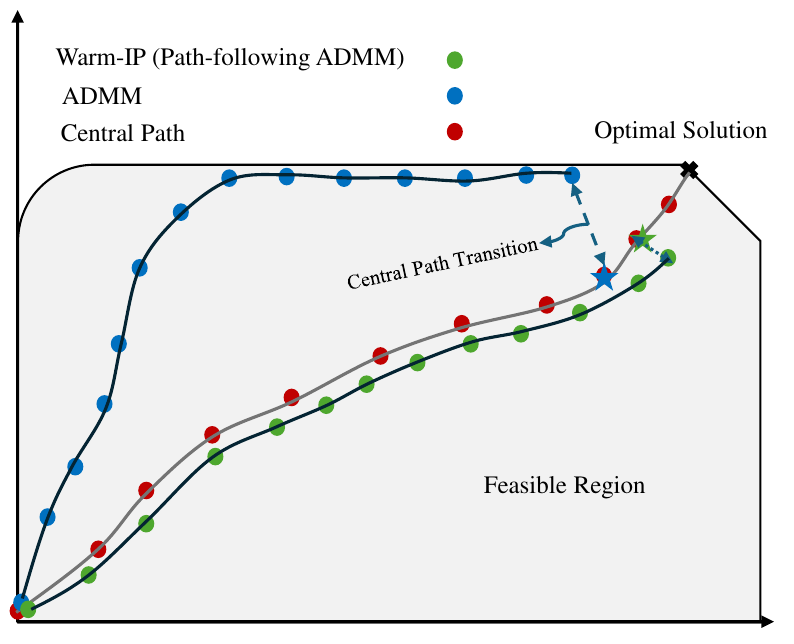}
		}

		\caption{\footnotesize Motivation for Warm-IP: (a) complementary convergence behavior of ADMM and IP ($\mathrm{KKT}_{\max}$, the maximum normalized KKT residual defined in Section~\ref{subsec:experimental_settings}, versus wall-clock time); (b) a naive ADMM warm start fails to accelerate IP, whereas a path-following ADMM warm start does; (c) standard ADMM iterates stray from the central path, while path-following ADMM iterates remain aligned with it.
		} 
		\label{conceptual}
	\end{figure}
	\subsection{Benchmark Dataset}
	\label{subsec:benchmark_dataset}
	
	To evaluate the proposed Warm-IP framework, we assembled a diverse benchmark suite spanning a broad range of problem sizes, structures, and application domains. Figure~\ref{data_size} summarizes the distribution of the QP instances in terms of the number of decision variables $n$ and constraints $m$, both shown on logarithmic scales. The benchmark suite includes established QP test sets, namely the Maros--Mészáros (MM) collection \citep{maros1999repository} and model predictive control (MPC) instances \citep{caron2024qpbenchmark}, as well as synthetically generated QPs and a new set of radiotherapy treatment planning QPs developed by our group.
	
	The MM collection is a standard test set with diverse numerical and structural properties, and the MPC instances arise from model predictive control applications; both were obtained from the public \href{https://github.com/qpsolvers/qpbenchmark}{QPBenchmark GitHub repository} \citep{caron2024qpbenchmark} and primarily cover the small- and medium-scale regimes.
	
	To complement these established test sets, we include synthetically generated QPs that provide broader coverage across problem dimensions; details of the generation procedure are provided in Section~\ref{sec:synthetic_benchmark}. We also curate a new benchmark set of large-scale QPs from cancer radiotherapy treatment planning, described in Section~\ref{sec:radiotherapy_benchmark}. These instances are derived from real patient treatment planning cases and are concentrated in the large-scale regime, where both the number of decision variables $n$ and the number of constraints $m$ are large.
	\begin{figure}[htbp]
		\centering
		\includegraphics{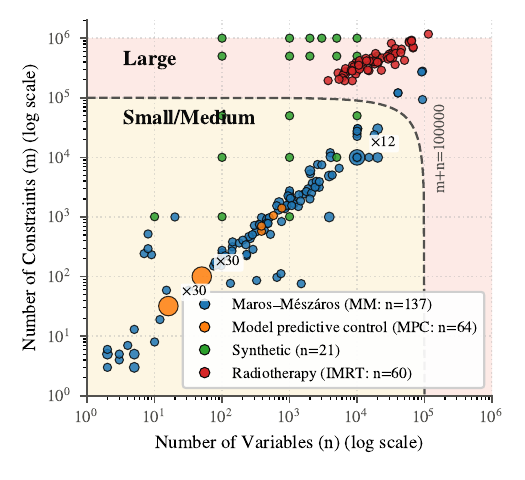}
		\caption{\footnotesize Benchmark instances by number of variables $n$ and constraints $m$. Markers merge instances with identical dimensions (marker area grows with the count; clusters of five or more are labeled); the dashed curve marks the threshold $m+n=100{,}000$ separating the small/medium and large regimes.}
		\label{data_size}
	\end{figure}
	The main contributions of this paper are summarized as follows.
	
	\begin{itemize}
		\item \textbf{A path-following ADMM.} We develop an ADMM variant for a log-barrier-modified QP and prove that, with the update order $x\rightarrow s\rightarrow z$, its iterates satisfy the central-path relation $SZe=\mu e$ exactly and remain strictly interior; producing well-centered warm starts by construction, rather than modifying the IP method to tolerate poorly centered ones.
		
		\item \textbf{The Warm-IP framework.} We integrate the path-following ADMM with a Mehrotra predictor--corrector IP method into a two-phase solver whose barrier parameter, stopping criteria, and linear-algebra structure are aligned across the phases.
		
		\item \textbf{A new large-scale benchmark.} We curate and publicly release 60 large-scale QP instances from lung-cancer radiotherapy treatment planning, built from real patient data with our open-source PortPy platform \citep{jhanwar2023portpy} and provided in canonical QP form so they can be used without radiotherapy knowledge, together with a complementary synthetic QP suite.
		
		\item \textbf{A comprehensive computational study.} We evaluate Warm-IP against two commercial and three open-source solvers across four benchmark classes; Warm-IP achieves the best performance on large-scale problems.
	\end{itemize}
	
	The remainder of the paper is organized as follows. Section~\ref{sec:methodology} presents the Warm-IP framework and its components. Section~\ref{numerical} reports the numerical experiments on the benchmark suite. Section~\ref{conclusion} provides concluding remarks and directions for future research.
	
	
	\section{Methodology}\label{sec:methodology}
	
	We consider a convex quadratic program and its Lagrange dual in the following form:
	\begin{equation}
		\label{eq:org_prob}
		\begin{aligned}
			\min_{x \in \mathbb{R}^n} \quad 
			& \frac{1}{2}x^\top Qx + q^\top x \\
			\text{s.t.} \quad 
			& Gx \le h, \\
			& Ax = b, \\
			& \textit{(primal)}
		\end{aligned}
		\qquad\qquad
		\begin{aligned}
			\max_{x,z,v} \quad 
			& -\frac{1}{2}x^\top Qx - h^\top z - b^\top v \\
			\text{s.t.} \quad 
			& Qx + G^\top z + A^\top v = -q, \\
			& z \ge 0, \\
			& \textit{(dual)}
		\end{aligned}
	\end{equation}
	Here, $Q\in\mathbb{S}_+^n$ is a symmetric positive semidefinite matrix, 
	$q\in\mathbb{R}^n$, $G\in\mathbb{R}^{m\times n}$, $h\in\mathbb{R}^m$, 
	$A\in\mathbb{R}^{p\times n}$, and $b\in\mathbb{R}^p$. The vector 
	$x\in\mathbb{R}^n$ denotes the primal decision variable, while 
	$z\in\mathbb{R}^m$ and $v\in\mathbb{R}^p$ are the Lagrange multipliers 
	associated with the inequality and equality constraints, respectively. The variable 
	$x$ appears in the dual formulation through the stationarity condition of the 
	Lagrangian. This primal--dual pair is standard; a detailed derivation of the dual 
	can be found in, e.g., \citet{boyd2004convex} and \citet{wright1997primal}.
	
	Throughout the paper, we make the following standing assumptions: (i) $A$ has full row rank; (ii) $Q$ is positive definite on the null space of $A$, or, alternatively, $\begin{bmatrix} G^\top & A^\top \end{bmatrix}^\top$ has full column rank; and (iii) both the primal and dual problems in \eqref{eq:org_prob} admit strictly feasible points, so that the central path is well defined \citep{boyd2004convex}. Assumptions (i)--(ii) guarantee that the KKT systems arising in the ADMM and IP phases, namely \eqref{eq:admm_xv_update} and \eqref{eq:reduced_kkt_remark}, are nonsingular.

	Introducing the slack variable $s$ for the inequality constraint $Gx \le h$, the Karush--Kuhn--Tucker (KKT) residual mapping is defined as:
	\begin{equation}
		\begin{aligned}
			F(x,v,z,s)
			&:=
			\begin{bmatrix}
				Qx + G^\top z + A^\top v + q \\
				Gx + s - h \\
				Ax - b \\
				SZe
			\end{bmatrix}
			&&
			\begin{array}{l}
				\text{(dual feasibility / stationarity)}\\
				\text{(primal feasibility: inequality constraints)}\\
				\text{(primal feasibility: equality constraints)}\\
				\text{(complementarity condition)}
			\end{array}
			\\[1ex]
		\end{aligned}
		\label{eq:KKT_residual}
	\end{equation}
	where $S=\operatorname{diag}(s)$, $Z=\operatorname{diag}(z)$, and $e$ is the vector of ones. The KKT optimality conditions are therefore given compactly by:
	\begin{equation}
		F(x,v,z,s)=0,
		\qquad s \ge 0,\; z \ge 0.
		\label{eq:KKT_compact}
	\end{equation}
	\subsection{Primal--Dual Interior-Point Algorithm}
	\label{sec:pdip}
	
	The primal--dual interior-point method is a well-established approach for solving convex optimization problems, and comprehensive derivations are available in standard references \citep{boyd2004convex,wright1997primal}. Here, we summarize only the essential algorithmic components used in our implementation. This review also provides the technical foundation for explaining how the path-following ADMM phase generates an effective warm start for the primal--dual interior-point method, referred to as IP in the remainder of the paper.
	
	The IP algorithm applies Newton's method to a perturbed form of the KKT system in \eqref{eq:KKT_compact}. Specifically, the complementarity condition $SZe=0$ is replaced by the relaxed condition $SZe = \mu e$, where $\mu>0$ is the barrier parameter and $e$ denotes the vector of ones. Starting from strictly positive slack and dual variables, $s>0$ and $z>0$, IP maintains positivity throughout the iterations while gradually reducing $\mu$ toward zero. In this way, the iterates follow the central path and converge to a solution satisfying the original KKT optimality conditions.

	At each iteration, the IP algorithm applies Newton's method to the perturbed KKT residual equations. Let $\omega := (x,v,z,s)$
	denote the collection of primal and dual variables. The Newton update has the following form:
	\begin{equation}
		\omega^{(k+1)} = \omega^{(k)} + \alpha^{(k)}\Delta \omega^{(k)},
		\label{eq:newton_update}
	\end{equation}
	where the Newton direction $\Delta \omega^{(k)}$ is obtained by solving
	\begin{equation}
		J\left(\omega^{(k)}\right)\Delta \omega^{(k)}
		=
		-F\left(\omega^{(k)}\right).
		\label{eq:newton_direction_general}
	\end{equation}
	Here, $F$ denotes the perturbed KKT residual mapping, $J(\omega)$ is its Jacobian with respect to the variables $(x,v,z,s)$, and $\alpha^{(k)}\in(0,1]$ is a step size selected by the fraction-to-the-boundary rule described in Appendix~\ref{app:step_size}.
	
	For the QP problem considered in this paper, the perturbed KKT residuals are
	\begin{equation}
		F(x,v,z,s)
		=
		\begin{bmatrix}
			r_{\mathrm d} \\
			r_{\mathrm p}^{\mathrm{ineq}} \\
			r_{\mathrm p}^{\mathrm{eq}} \\
			r_{\mathrm c}
		\end{bmatrix}
		=
		\begin{bmatrix}
			Qx + A^\top v + G^\top z + q \\
			Gx + s - h \\
			Ax - b \\
			SZe - \mu e
		\end{bmatrix},
		\label{eq:perturbed_kkt_residual}
	\end{equation}
	where $\mu\ge 0$ is the target complementarity parameter of the perturbed system and $m$ denotes the number of inequality constraints. The value of $\mu$ is chosen adaptively at each iteration, guided by the current average complementarity $\frac{s^\top z}{m}$: the Mehrotra predictor step uses $\mu=0$, while the corrector step uses $\mu=\sigma\,\frac{s^\top z}{m}$ with the centering parameter $\sigma$ defined in Appendix~\ref{app:mehrotra}. Therefore, each IP iteration requires solving the following sparse linear system:
	\begin{equation}
		\underbrace{
			\begin{bmatrix}
				Q & A^\top & G^\top & 0 \\
				G & 0 & 0 & I \\
				A & 0 & 0 & 0 \\
				0 & 0 & S & Z
			\end{bmatrix}
		}_{J(x,v,z,s)}
		\underbrace{
			\begin{bmatrix}
				\Delta x \\
				\Delta v \\
				\Delta z \\
				\Delta s
			\end{bmatrix}
		}_{\Delta \omega}
		=
		-
		\underbrace{
			\begin{bmatrix}
				r_{\mathrm d} \\
				r_{\mathrm p}^{\mathrm{ineq}} \\
				r_{\mathrm p}^{\mathrm{eq}} \\
				r_{\mathrm c}
			\end{bmatrix}
		}_{F(x,v,z,s)}.
		\label{eq:kkt_system}
	\end{equation}

	After solving \eqref{eq:kkt_system}, the primal and dual variables are updated along the Newton direction:
	\begin{equation}
		(x^{(k+1)},s^{(k+1)},v^{(k+1)},z^{(k+1)})
		=
		(x^{(k)},s^{(k)},v^{(k)},z^{(k)})
		+
		\alpha(\Delta x,\Delta s,\Delta v,\Delta z).
		\label{eq:pdip_update}
	\end{equation}
	The step size $\alpha \in (0,1]$ is chosen to ensure that the next iterate remains strictly interior.
	The following remarks summarize implementation details that improve the efficiency and numerical stability of the IP solver.

	\begin{remark}
		\label{rem:reduced_kkt}
		\textit{Reduced KKT system.}
		The full Newton system in \eqref{eq:kkt_system} can be reduced by block elimination, equivalently by applying a Schur complement argument \citep{boyd2004convex}. In particular, the directions $\Delta s$ and $\Delta z$ can be eliminated, yielding a smaller saddle-point system in the variables $\Delta x$ and $\Delta v$:
		\begin{equation}
			\begin{bmatrix}
				K & A^\top \\
				A & 0
			\end{bmatrix}
			\begin{bmatrix}
				\Delta x \\
				\Delta v
			\end{bmatrix}
			=
			\begin{bmatrix}
				\mathrm{rhs}_x \\
				\mathrm{rhs}_v
			\end{bmatrix},
			\label{eq:reduced_kkt_remark}
		\end{equation}
		where
		\begin{equation}
			K = Q + G^\top D G,
			\qquad
			D = ZS^{-1} = \operatorname{diag}(z_i/s_i).
			\label{eq:K_D_definition}
		\end{equation}
		This reduction decreases the dimension of the linear system, which is central to the efficiency of the implementation. The detailed derivation is provided in Appendix~\ref{sec:reduced_kkt}.
	\end{remark}
	
	\begin{remark}
		\label{rem:sparsity_reuse}
		\textit{Reuse of symbolic factorization.}
		For QP, the sparsity pattern of the reduced KKT matrix in \eqref{eq:reduced_kkt_remark} remains unchanged across IP iterations; only the numerical values of some matrix entries change. Therefore, the ordering and symbolic factorization computed during the first linear solve can be reused in subsequent iterations. This avoids repeated symbolic analysis and reduces the computational overhead of sparse matrix factorization.
	\end{remark}

	\begin{remark}
		\label{rem:stopping_criteria}
		\textit{Stopping criteria.} The IP iterations are terminated when the primal and dual feasibility residuals and the duality gap are all below a prescribed tolerance $\varepsilon>0$. To make the stopping criterion robust to problem scaling, we use normalized primal and dual residuals:
		\begin{equation}
			r_{\mathrm p}^{\mathrm{ineq}}
			=
			\frac{\|Gx+s-h\|_\infty}{1+\max\{\|Gx\|_\infty,\|s\|_\infty,\|h\|_\infty\}},
			\qquad
			r_{\mathrm p}^{\mathrm{eq}}
			=
			\frac{\|Ax-b\|_\infty}{1+\max\{\|Ax\|_\infty,\|b\|_\infty\}},
		\end{equation}
		and define
		\begin{equation}
			r_{\mathrm p}
			=
			\max\left\{
			r_{\mathrm p}^{\mathrm{ineq}},
			r_{\mathrm p}^{\mathrm{eq}}
			\right\}.
		\end{equation}
		Similarly, the normalized dual residual is
		\begin{equation}
			r_{\mathrm d}
			=
			\frac{\|Qx+G^\top z+A^\top v+q\|_\infty}
			{1+\max\{\|Qx\|_\infty,\|G^\top z\|_\infty,\|A^\top v\|_\infty,\|q\|_\infty\}}.
		\end{equation}
		Finally, optimality is measured by the normalized duality gap
		\begin{equation}
			r_{\mathrm{gap}}
			=
			\frac{\lvert p-d\rvert}{1+\max\{\lvert p\rvert,\lvert d\rvert\}},
			\qquad
			p=\tfrac{1}{2}x^\top Qx+q^\top x,
			\qquad
			d=-\tfrac{1}{2}x^\top Qx-h^\top z-b^\top v,
		\end{equation}
		where $p$ and $d$ are the primal and dual objective values at the current iterate; we use the symbol $r_{\mathrm{gap}}$ to distinguish this scalar quantity from the complementarity residual vector $r_{\mathrm c}$ in \eqref{eq:perturbed_kkt_residual}. Convergence is declared when
		\begin{equation}
			r_{\mathrm p}\le \varepsilon,
			\qquad
			r_{\mathrm d}\le \varepsilon,
			\qquad
			r_{\mathrm{gap}}\le \varepsilon.
			\label{eq:stopping_criteria}
		\end{equation}
	\end{remark}
	
		Algorithm~\ref{alg:pdip} in Appendix~\ref{app:pdip_alg} assembles these ingredients into the complete primal--dual interior-point method. It requires an initial interior point, which the path-following ADMM method introduced in the next section is designed to provide. For completeness, Appendix~\ref{sec:app_initialization} also describes a commonly used heuristic initialization strategy, employed for example in CVXOPT \citep{andersen2012interior} and Clarabel \citep{goulart2024clarabel}. This heuristic is included in our Python implementation so that IP can also be run as a standalone solver without ADMM initialization.

		
		\subsection{Path-following ADMM}
		
		A key challenge in using ADMM to warm-start an IP method is that the two algorithms approach the solution through different regions of the primal--dual space. In IP, the iterates are maintained strictly interior, with $s>0$ and $z>0$, and the complementarity condition is relaxed as $SZe \approx \mu e$, where the barrier parameter $\mu>0$ is gradually reduced to zero. Thus, IP iterates are designed to remain close to the central path while simultaneously reducing the primal and dual residuals.
		
		In contrast, standard ADMM is not designed to follow the central path. Although ADMM can efficiently reduce primal and dual residuals, its iterates do not generally satisfy an interior-point complementarity relation of the form $SZe\approx \mu e$. Consequently, an ADMM iterate may be close to optimality in terms of feasibility or objective value, but still be poorly aligned with the central path. Such a point may provide a weak warm start for IP, because the interior-point method may need several iterations to move the iterate back toward a central-path neighborhood before benefiting from its fast local convergence.
		
		To address this mismatch, we introduce a path-following variant of ADMM. The main idea is to modify the primal problem by adding a logarithmic barrier term for the slack variables. This construction is inspired by primal interior-point methods, where a log-barrier term is added to the objective and its weight is gradually reduced. In our approach, however, the barrier weight is not driven to zero during the ADMM phase: it is held at a level $\mu$ matched to the ADMM accuracy target $\varepsilon_{\mathrm{ADMM}}$ (constant within each iteration, which is all the analysis requires; the precise choice is given in Section~\ref{subsec:adaptive_rho}). We then run ADMM until the primal and dual residuals fall below $\varepsilon_{\mathrm{ADMM}}$. Because the barrier pins the complementarity level at $\mu$, the resulting point is approximately feasible and stationary while remaining aligned with the central-path neighborhood associated with $\mu$. This path-following ADMM solution is then used to initialize IP, which continues the optimization by further reducing $\mu$ and refining the solution to high accuracy.
		
		Specifically, we apply ADMM to the following log-barrier-modified QP and its dual (whose objective is stated up to the additive constant $m\mu(1-\log\mu)$):
		\begin{equation}
			\label{eq:org_prob_modified}
			\begin{aligned}
				\textit{(primal)} \qquad \min_{x \in \mathbb{R}^n,\; s} \quad & \frac{1}{2}x^\top Qx + q^\top x - \mu \sum_{i=1}^{m} \log (s_i) \\
				\text{s.t.} \quad & Gx + s = h, \quad Ax = b, \quad s > 0, \\[4pt]
				\textit{(dual)} \qquad \max_{x,z,v} \quad & -\frac{1}{2}x^\top Qx - h^\top z - b^\top v + \mu \sum_{i=1}^{m} \log(z_i) \\
				\text{s.t.} \quad & Qx + G^\top z + A^\top v = -q, \quad z > 0.
			\end{aligned}
		\end{equation}
		The KKT conditions of the log-barrier-modified problem have the same primal and dual feasibility equations as the original QP, while the complementarity condition is replaced by $SZe=\mu e$; the derivation of this dual and of these optimality conditions is given in Appendix~\ref{sec:app:dual_derivation}.
		
		This is precisely the central-path relation used by IP for a fixed barrier parameter $\mu$. In the following, we show that when ADMM is applied to \eqref{eq:org_prob_modified} with a particular ordering of the variable updates, the iterates satisfy this perturbed complementarity condition throughout the ADMM phase. Therefore, if ADMM is run until the primal and dual residuals reach $\varepsilon_{\mathrm{ADMM}}$, with $\mu$ chosen so that the associated duality gap sits at the same level (Section~\ref{subsec:adaptive_rho}), the resulting point has balanced residuals: it is approximately primal feasible, approximately dual feasible, and consistent with the central-path relation at $\mu$. This makes it an effective warm start for IP, which can then continue from this point by decreasing $\mu$ and converging to a high-accuracy solution.
		
		
		\subsubsection{Brief Overview of ADMM}
		\label{subsec:admm_overview}
		
		ADMM is a first-order operator-splitting method designed for optimization problems with separable structure and coupling constraints. In its standard two-block form, ADMM solves problems of the form
		\begin{equation}
			\begin{aligned}
				\min_{u,w} \quad & f(u)+g(w) \\
				\text{s.t.} \quad & Bu+Cw=d,
			\end{aligned}
			\label{eq:admm_standard_form}
		\end{equation}
		where the objective is separable in the variables $u$ and $w$, while the constraint couples them. For a penalty parameter $\rho>0$ and multiplier $\lambda$, the augmented Lagrangian is
		\begin{equation}
			\mathcal{L}_{\rho}(u,w,\lambda)
			=
			f(u)+g(w)
			+\lambda^\top(Bu+Cw-d)
			+\frac{\rho}{2}\|Bu+Cw-d\|_2^2 .
			\label{eq:admm_aug_lagrangian_general}
		\end{equation}
		ADMM then alternates between minimizing this augmented Lagrangian with respect to one block of variables while keeping the other fixed, followed by a multiplier update:
		\begin{equation}
			\begin{aligned}
				u^{(k+1)}
				&=
				\arg\min_u \mathcal{L}_{\rho}(u,w^{(k)},\lambda^{(k)}),\\
				w^{(k+1)}
				&=
				\arg\min_w \mathcal{L}_{\rho}(u^{(k+1)},w,\lambda^{(k)}),\\
				\lambda^{(k+1)}
				&=
				\lambda^{(k)}
				+
				\rho\left(Bu^{(k+1)}+Cw^{(k+1)}-d\right).
			\end{aligned}
			\label{eq:admm_general_updates}
		\end{equation}
		
		In the log-barrier-modified QP considered here, the two blocks correspond to the primal decision variable $x$ and the slack variable $s$. The coupling constraint is $Gx+s=h,$ and the equality constraint $Ax=b$ is enforced directly in the $x$-update. Thus, ADMM alternates between an equality-constrained quadratic update for $x$, a separable log-barrier update for $s$, and a multiplier update for $z$, the dual variable associated with $Gx+s=h$. This splitting is particularly useful in our setting because the $x$-update reduces to solving a structured linear system, while the $s$-update has a closed-form componentwise solution. Most importantly, as shown below, when the variables are updated in the order $x \rightarrow s \rightarrow z$, the resulting iterates satisfy the perturbed complementarity relation $SZe=\mu e,$ which aligns the ADMM warm start with the central-path geometry of IP.
		
		We now specialize these general ADMM updates to the log-barrier-modified QP in \eqref{eq:org_prob_modified}.
		
		
		\subsubsection{ADMM Updates for the Log-Barrier-Modified Problem}
		\label{subsec:path_following_admm_updates}
		
		For a fixed barrier parameter $\mu>0$ and ADMM penalty parameter $\rho>0$, we split the variables into two blocks, $x$ and $s$, coupled through the constraint $Gx+s=h$; as above, $Ax=b$ is enforced directly in the $x$-update, and $z$ denotes the multiplier of the coupling constraint. The augmented Lagrangian is
		\begin{equation}
			\begin{aligned}
				\mathcal{L}_{\rho}(x,s,z)
				&=
				\frac{1}{2}x^\top Qx+q^\top x
				-\mu\sum_{i=1}^m \log(s_i)
				+z^\top(Gx+s-h)
				+\frac{\rho}{2}\|Gx+s-h\|_2^2 .
			\end{aligned}
			\label{eq:admm_aug_lagrangian}
		\end{equation}
		
		At iteration $k$, given $(s^{(k)},z^{(k)})$, the $x$-update is
		\begin{equation}
			\begin{aligned}
				x^{(k+1)}
				=
				\arg\min_{x}
				&\left\{
				\frac{1}{2}x^\top Qx+q^\top x
				+
				(z^{(k)})^\top(Gx+s^{(k)}-h)
				+
				\frac{\rho}{2}\|Gx+s^{(k)}-h\|_2^2
				\right\} \\
				&\;\text{s.t. } Ax=b .
			\end{aligned}
			\label{eq:admm_x_update}
		\end{equation}
		This is an equality-constrained quadratic problem. Its optimality conditions give the KKT system
		\begin{equation}
			\begin{bmatrix}
				Q+\rho G^\top G & A^\top\\
				A & 0
			\end{bmatrix}
			\begin{bmatrix}
				x^{(k+1)}\\
				v^{(k+1)}
			\end{bmatrix}
			=
			\begin{bmatrix}
				-q-G^\top z^{(k)}+\rho G^\top(h-s^{(k)})\\
				b
			\end{bmatrix},
			\label{eq:admm_xv_update}
		\end{equation}
		where $v^{(k+1)}$ is the multiplier associated with the equality constraint $Ax=b$.
		
		\begin{remark}
			\label{rem:admm_factorization_reuse}
			\textit{Factorization reuse.}
			For fixed $\rho$, the KKT matrix in the ADMM $x$-update \eqref{eq:admm_xv_update} remains unchanged across ADMM iterations; only the right-hand side changes. Hence, this matrix can be factorized once and reused in all subsequent ADMM iterations, making later iterations relatively inexpensive. In addition, this matrix has the same sparsity pattern as the reduced IP KKT matrix in \eqref{eq:reduced_kkt_remark},
			because $D$ is diagonal and both $G^\top D G$ and $G^\top G$ share the same sparsity structure. Therefore, the symbolic factorization and ordering computed during the ADMM phase can be reused in the subsequent IP phase, further reducing the cost of the warm-started method.
		\end{remark}
		
		Next, with $x^{(k+1)}$ fixed, the slack variable is updated by solving
		\begin{equation}
			s^{(k+1)}
			=
			\arg\min_{s>0}
			\left\{
			-\mu\sum_{i=1}^m \log(s_i)
			+
			(z^{(k)})^\top(Gx^{(k+1)}+s-h)
			+
			\frac{\rho}{2}\|Gx^{(k+1)}+s-h\|_2^2
			\right\}.
			\label{eq:admm_s_update}
		\end{equation}
		This problem is separable across the components of $s$. For each component $i$, the first-order optimality condition is
		\begin{equation}
			-\frac{\mu}{s_i}
			+
			z_i^{(k)}
			+
			\rho\left(G_i x^{(k+1)}+s_i-h_i\right)
			=0,
			\label{eq:s_update_optimality}
		\end{equation}
		where $G_i$ denotes the $i$th row of $G$. Rearranging \eqref{eq:s_update_optimality} gives the quadratic equation
		\begin{equation}
			\rho s_i^2
			+
			\left(z_i^{(k)}+\rho(G_i x^{(k+1)}-h_i)\right)s_i
			-
			\mu
			=
			0.
			\label{eq:s_update_quadratic}
		\end{equation}
		Since the slack variable must satisfy $s_i>0$, we take the positive root. For compact notation, define
		\begin{equation}
			r_1^{(k+1)}
			=
			h-Gx^{(k+1)}-\frac{z^{(k)}}{\rho},
			\qquad
			r_2
			=
			\frac{2\mu}{\rho}.
			\label{eq:s_update_r1_r2}
		\end{equation}
		Then the slack update can be written componentwise as
		\begin{equation}
			s_i^{(k+1)}
			=
			\frac{1}{2}
			\left(
			r_{1,i}^{(k+1)}
			+
			\sqrt{
				\left(r_{1,i}^{(k+1)}\right)^2
				+
				2r_2
			}
			\right),
			\qquad i=1,\ldots,m.
			\label{eq:admm_s_closed_form}
		\end{equation}
		The following lemma shows that the slack update automatically maintains $s^{(k+1)}>0$.

		\begin{lemma}
			\label{lem:admm_slack_positivity}
			For fixed $\mu>0$ and $\rho>0$, the slack iterates produced by \eqref{eq:admm_s_closed_form} are strictly positive at every ADMM iteration.
		\end{lemma}
		
		\noindent\textbf{Proof.}
		From \eqref{eq:admm_s_closed_form}, each component of the slack update is given by
		\[
		s_i^{(k+1)}
		=
		\frac{1}{2}
		\left(
		r_{1,i}^{(k+1)}
		+
		\sqrt{
			\left(r_{1,i}^{(k+1)}\right)^2+2r_2
		}
		\right),
		\]
		where $r_2=2\mu/\rho$. Since $\mu>0$ and $\rho>0$, we have $r_2>0$. Therefore,
		\[
		\sqrt{
			\left(r_{1,i}^{(k+1)}\right)^2+2r_2
		}
		>
		\left|r_{1,i}^{(k+1)}\right|.
		\]
		It follows that
		\[
		r_{1,i}^{(k+1)}
		+
		\sqrt{
			\left(r_{1,i}^{(k+1)}\right)^2+2r_2
		}
		>
		r_{1,i}^{(k+1)}
		+
		\left|r_{1,i}^{(k+1)}\right|.
		\]
		The right-hand side is always nonnegative, since
		\[
		r_{1,i}^{(k+1)}+\left|r_{1,i}^{(k+1)}\right|
		=
		\begin{cases}
			2r_{1,i}^{(k+1)}, & r_{1,i}^{(k+1)}\ge 0,\\
			0, & r_{1,i}^{(k+1)}<0.
		\end{cases}
		\]
		Hence,
		\[
		r_{1,i}^{(k+1)}
		+
		\sqrt{
			\left(r_{1,i}^{(k+1)}\right)^2+2r_2
		}
		>0,
		\]
		and therefore $s_i^{(k+1)}>0$ for all $i=1,\ldots,m$. This completes the proof. \hfill$\square$

		Finally, after updating $x$ and $s$, the multiplier for the coupling constraint $Gx+s=h$ is updated as
		\begin{equation}
			z^{(k+1)}
			=
			z^{(k)}
			+
			\rho\left(Gx^{(k+1)}+s^{(k+1)}-h\right).
			\label{eq:admm_z_update}
		\end{equation}
		
		The following theorem shows that this update order preserves the perturbed complementarity relation associated with the central path.
		\begin{theorem}
			\label{thm:admm_complementarity}
			Consider the ADMM updates \eqref{eq:admm_x_update}--\eqref{eq:admm_z_update}
			applied to the log-barrier-modified problem \eqref{eq:org_prob_modified}
			with fixed $\mu>0$ and $\rho>0$. If the variables are updated in the order $x \;\rightarrow\; s \;\rightarrow\; z$, then, for every ADMM iteration $k$,
			\begin{equation}
				S^{(k+1)}Z^{(k+1)}e=\mu e.
				\label{eq:admm_complementarity_matrix}
			\end{equation}
		\end{theorem}
		
		\noindent\textbf{Proof.}
		Fix an ADMM iteration $k$. The first-order optimality condition of the slack update \eqref{eq:admm_s_update}, given in \eqref{eq:s_update_optimality}, evaluated at $s^{(k+1)}$, gives
		\[
		-\frac{\mu}{s_i^{(k+1)}}
		+
		z_i^{(k)}
		+
		\rho\left(G_i x^{(k+1)}+s_i^{(k+1)}-h_i\right)
		=0.
		\]
		By the multiplier update \eqref{eq:admm_z_update}, the last two terms are exactly $z_i^{(k+1)}$. Hence,
		\[
		-\frac{\mu}{s_i^{(k+1)}}+z_i^{(k+1)}=0,
		\]
		which implies
		\[
		s_i^{(k+1)}z_i^{(k+1)}=\mu,
		\qquad i=1,\ldots,m.
		\]
		Equivalently, $S^{(k+1)}Z^{(k+1)}e=\mu e$. \hfill$\square$
		
		\begin{corollary}
			\label{cor:strict_interiority}
			Under the assumptions of Theorem~\ref{thm:admm_complementarity}, $z^{(k+1)}=\mu\, (S^{(k+1)})^{-1}e>0$ componentwise. Combined with Lemma~\ref{lem:admm_slack_positivity}, the ADMM iterates therefore remain strictly interior, i.e., $s^{(k)}>0$ and $z^{(k)}>0$ for all $k\ge 1$.
		\end{corollary}
		
		Theorem~\ref{thm:admm_complementarity} shows that, with the update order $x\rightarrow s\rightarrow z$, the ADMM iterates satisfy the same perturbed complementarity relation used by IP for the fixed barrier parameter $\mu$. Therefore, once the primal and dual residuals have been reduced to the tolerance $\varepsilon_{\mathrm{ADMM}}$, the resulting ADMM iterate is approximately feasible, approximately stationary, and already aligned with the central-path geometry. This makes it a natural warm start for the subsequent IP phase.
		
		
		\subsubsection{Stopping Criteria and Adaptive Barrier Weight}
		\label{subsec:adaptive_rho}
		
		The ADMM penalty parameter $\rho$ is updated adaptively to balance the primal and dual residuals, following the residual-balancing strategy of \citet{he2000alternating} as used in OSQP \citep{stellato2020osqp}; the rule, and why it preserves the complementarity relation of Theorem~\ref{thm:admm_complementarity}, are given in Appendix~\ref{app:adaptive_rho}.
		For path-following ADMM, we use stopping criteria analogous to those of IP in Remark~\ref{rem:stopping_criteria}, but without explicitly checking the duality gap. The gap does not need to be monitored because it is known in advance. By Theorem~\ref{thm:admm_complementarity}, every ADMM iterate satisfies $SZe=\mu e$, so $s^\top z=m\mu$ at all times. Moreover, at any iterate the duality gap satisfies the exact identity $p-d = s^\top z + x^\top(Qx+G^\top z+A^\top v+q) - z^\top(Gx+s-h) - v^\top(Ax-b)$, where the parenthesized vectors are the stationarity and primal feasibility residuals of \eqref{eq:KKT_residual} (see Appendix~\ref{sec:app:dual_derivation}); once the primal and dual residuals are small, the residual terms are negligible and $p-d\approx s^\top z=m\mu$. The barrier weight is chosen in the remark below precisely so that this known gap, in normalized form, equals the tolerance $\varepsilon_{\mathrm{ADMM}}$. The ADMM phase is therefore terminated once the primal and dual residuals are sufficiently small, with the duality gap landing at the target level automatically.
		
		\begin{remark}
			\label{rem:adaptive_mu}
			\textit{Choice of the barrier weight.} The analysis above holds for any barrier weight; what remains is to choose it. As just noted, the duality gap at the handoff equals $s^\top z=m\mu$, which in the normalized form of Remark~\ref{rem:stopping_criteria} reads $m\mu/(1+\max\{\lvert p\rvert,\lvert d\rvert\})$, where $p$ and $d$ are the primal and dual objective values at the current iterate. Setting this equal to $\varepsilon_{\mathrm{ADMM}}$ determines the barrier weight:
			\begin{equation}
				\mu
				=
				\varepsilon_{\mathrm{ADMM}}\,\frac{1+\max\{\lvert p\rvert,\lvert d\rvert\}}{m}.
				\label{eq:mu_adaptive}
			\end{equation}
			With this choice, the ADMM phase terminates at a point where the primal residual, the dual residual, and the normalized duality gap are all approximately $\varepsilon_{\mathrm{ADMM}}$: a balanced, well-centered warm start at exactly the accuracy level requested of the ADMM phase. Because $p$ and $d$ are not known in advance, the weight is re-evaluated at each residual check using the current objective values. Since $\mu$ enters only the closed-form $s$-update \eqref{eq:admm_s_closed_form}, updating it requires no new matrix factorization.
		\end{remark}
		
		Algorithm~\ref{alg:path_following_admm} summarizes the path-following ADMM method. The returned iterate $(x,v,z,s)$ directly initializes the IP method of Section~\ref{sec:pdip}, completing the two-phase Warm-IP solver.

		\begin{algorithm}[t]
			\small
			\caption{Path-Following ADMM for Warm-Starting IP}
			\label{alg:path_following_admm}
			\begin{algorithmic}[1]
				\Require QP data $(Q,q,G,h,A,b)$, initial penalty $\rho>0$, tolerance $\varepsilon_{\mathrm{ADMM}}$, $\rho$ update threshold $\tau_{\rho}$
				\Ensure Warm-start point $(x,v,z,s)$ for IP
				
				\State Initialize $s^{(0)}\gets 0$, $z^{(0)} \gets 0$, and $r_{\mathrm p}^{(0)}, r_{\mathrm d}^{(0)} \gets \infty$
				\State Set $k\gets 0$ and $\mu\gets \varepsilon_{\mathrm{ADMM}}$
				
				\While{$r_{\mathrm p}^{(k)}>\varepsilon_{\mathrm{ADMM}}$ or $r_{\mathrm d}^{(k)}>\varepsilon_{\mathrm{ADMM}}$}
				
				\State \textbf{$x$-update:} Solve the KKT system \eqref{eq:admm_xv_update} for $(x^{(k+1)},v^{(k+1)})$
				\Statex \hspace{\algorithmicindent}{\footnotesize$\triangleright$ reuse the previous factorization (see Remark~\ref{rem:admm_factorization_reuse}).}
				\State \textbf{$s$-update:} Compute $s^{(k+1)}$ using the closed-form update \eqref{eq:admm_s_closed_form}
				\Statex \hspace{\algorithmicindent}{\footnotesize$\triangleright$ $s^{(k+1)}>0$ by Lemma~\ref{lem:admm_slack_positivity}.}
				
				\State \textbf{$z$-update:} Update the multiplier $z^{(k+1)}$ using \eqref{eq:admm_z_update}
				\Statex \hspace{\algorithmicindent}{\footnotesize$\triangleright$ the order $x\rightarrow s\rightarrow z$ ensures $S^{(k+1)}Z^{(k+1)}e=\mu e$ by Theorem~\ref{thm:admm_complementarity}.}
				
				\State Compute the ADMM primal and dual residuals $r_{\mathrm p}^{(k+1)}$ and $r_{\mathrm d}^{(k+1)}$
				\Statex \hspace{\algorithmicindent}{\footnotesize$\triangleright$ use the normalized versions in Remark~\ref{rem:stopping_criteria}.}
				\State Update $\mu$ by the adaptive rule \eqref{eq:mu_adaptive}
				\Statex \hspace{\algorithmicindent}{\footnotesize$\triangleright$ $\mu$ enters only the $s$-update \eqref{eq:admm_s_closed_form}; no refactorization is needed (Remark~\ref{rem:adaptive_mu}).}
				\State Optionally update $\rho$ using the adaptive rule \eqref{eq:rho_scale}--\eqref{eq:rho_update} (Appendix~\ref{app:adaptive_rho})
				\Statex \hspace{\algorithmicindent}{\footnotesize$\triangleright$ If $\rho$ changes, refactorize the numerical KKT matrix; reuse the previous symbolic factorization by Remark~\ref{rem:admm_factorization_reuse}.}
				
				\State Set $k\gets k+1$
				
				\EndWhile
				
				\State \Return $(x^{(k)},v^{(k)},z^{(k)},s^{(k)})$
				
			\end{algorithmic}
		\end{algorithm}

		\section{Numerical Experiments}
		\label{numerical}
		
		This section presents the computational evaluation of Warm-IP across four benchmark classes: model predictive control (MPC), Maros--Mészáros (MM), synthetic QPs, and intensity-modulated radiation therapy (IMRT) treatment planning problems. We compare Warm-IP against five state-of-the-art solvers: two commercial solvers, \textit{MOSEK} \citep{aps2022mosek} and \textit{Gurobi} \citep{gurobi}, and three open-source solvers, \textit{Clarabel} \citep{goulart2024clarabel}, \textit{PIQP} \citep{schwan2023piqp}, and \textit{OSQP} \citep{stellato2020osqp}. The remainder of this section is organized as follows. We first describe the experimental settings, including the computing environment, solver interfaces, and tolerance settings. We then describe the synthetic and radiotherapy benchmark sets. Finally, we present solver comparisons for each benchmark class, followed by aggregate analyses across problem sizes.
		
		
		\subsection{Experimental Settings}
		\label{subsec:experimental_settings}
		
		All experiments were conducted on a desktop workstation equipped with two Intel\textsuperscript{\textregistered} Xeon\textsuperscript{\textregistered} Gold 5420+ processors (2.00~GHz) and 256~GB of RAM. All implementations were run in Python 3.12 on the Windows 10 operating system. Reported runtimes are wall-clock solve times, excluding the conversion of each instance to the solver's input format; for Warm-IP they include both the ADMM and the IP phase.
		
		Warm-IP is implemented entirely in Python (publicly available in the \href{https://github.com/Radiotherapy-Optimization/Warm-IP}{Warm-IP GitHub repository}), which favors transparency and extensibility but incurs overhead relative to the compiled C/C++/Rust cores underlying the benchmark solvers.
		
		We benchmarked against MOSEK 11.2.3 (accessed through CVXPY), Gurobi 13.0.2, Clarabel 0.11.1, PIQP 0.6.3, and OSQP 1.1.3, each called through its Python interface. The termination tolerances of all solvers were set to $10^{-6}$; every other solver parameter, including the iteration limits and all preprocessing, scaling, and ordering options, was left at its default value. In addition, every solve (including Warm-IP and IP) was subject to a wall-clock time limit of $3{,}600$ seconds on the IMRT and synthetic benchmarks and $360$ seconds on the MM and MPC benchmarks; a solver exceeding the limit was terminated and counted as failed on that instance (a terminated solver returns no solution).
		
		For Warm-IP, the path-following ADMM phase was terminated when its primal and dual residuals fell below $\varepsilon_{\mathrm{ADMM}}=0.5$ (Section~\ref{subsec:adaptive_rho}), and the IP phase at tolerance $\varepsilon_{\mathrm{IP}}=10^{-6}$; the Newton systems were factorized with the Intel MKL PARDISO sparse direct solver. Because many MM instances are poorly scaled, the Ruiz equilibration procedure implemented in OSQP \citep{stellato2020osqp} was applied to the MM problems as a preprocessing step for Warm-IP and IP (the external solvers received the original data and apply their own internal scaling; all accuracy metrics are evaluated on the original data).
		
		All solvers were evaluated with the same accuracy metrics. Given the primal--dual solution returned by a solver, we computed the normalized dual residual and normalized duality gap of Remark~\ref{rem:stopping_criteria}, together with the normalized primal residual in \emph{constraint-violation form}---with $\|\max(Gx-h,0)\|_\infty$ in place of $\|Gx+s-h\|_\infty$ in the numerator of $r_{\mathrm p}^{\mathrm{ineq}}$, so that solvers are not penalized for inactive slack---and summarized them by their maximum, denoted $\mathrm{KKT}_{\max}$. Warm-IP and IP return the best iterate encountered by this measure. A solver is considered to have \emph{failed} on an instance when it returns no solution or when $\mathrm{KKT}_{\max}>10^{-3}$.
		
		Equality constraints were passed to the external solvers in their native form $Ax=b$, but converted to pairs of opposing inequalities for Warm-IP and IP ($Ax=b \;\Rightarrow\; Ax\le b$ and $-Ax\le -b$): the reduced KKT systems then remain positive definite and factorize robustly, whereas the (also implemented) direct treatment of $Ax=b$ yields indefinite saddle-point systems that are more delicate to factorize.

		
		\subsection{Synthetic Benchmark Generation}
		\label{sec:synthetic_benchmark}
		To supplement the standard benchmark sets and obtain broader coverage across problem dimensions, we generated a collection of synthetic convex QP instances. Each instance has the form
		\[
		\min_x \quad \frac{1}{2}x^\top Qx + q^\top x
		\quad \text{s.t.} \quad Ax=b,\qquad Gx\le h.
		\]
		For a given number of variables $n$, equality constraints $m_{\mathrm{eq}}$, and inequality constraints $m_{\mathrm{ineq}}$, we generate a sparse random $n\times n$ matrix $M$ (density $0.01$) and set $Q = M^\top M + \delta I$ with $\delta=1$, so that $Q$ is positive definite and remains sparse, which is what allows the largest instances to be generated and solved. The constraint matrices $A$ and $G$ are sparse random matrices with densities $0.01$ and $0.001$, respectively. The nonzero entries of $M$, $A$, and $G$ are placed uniformly at random and drawn uniformly from $[0,1)$. Conditioning is controlled by an ill-conditioning parameter $\kappa\ge0$: the columns of $M$ are rescaled by log-uniform factors $10^{\kappa u_j}$ with $u_j\sim U[0,1]$, and the rows of $A$ and $G$ are rescaled by $10^{\kappa u_i}$ with $u_i\sim U[-1,1]$, spreading the constraint-row norms over $10^{\pm\kappa}$; this mimics the heterogeneous scaling found in real-world instances. We then draw a reference point $x_0$ with independent standard normal entries, set $q=-Qx_0$ (making $x_0$ the unconstrained minimizer of the objective), and compute the right-hand sides after rescaling as $b=Ax_0$ and $h=Gx_0+\xi$ with $\xi$ a strictly positive random vector whose entries are drawn independently and uniformly from $(0,1)$, so that $x_0$ is feasible with strictly positive inequality slacks. Varying $n$, $m_{\mathrm{eq}}$, and $m_{\mathrm{ineq}}$ yields controlled instances across problem sizes and variable-to-constraint ratios. All synthetic instances in our experiments use $\kappa=1$.
		
		
		\subsection{Intensity-Modulated Radiation Therapy (IMRT) Benchmark}
		\label{sec:radiotherapy_benchmark}
		
		Radiotherapy is a major cancer treatment modality in which high-energy radiation beams are delivered to a tumor while minimizing radiation exposure to nearby healthy organs. Figure~\ref{fig:rt_overview} illustrates the main components of this process. A radiotherapy machine delivers radiation from multiple beam directions, and the shape of each beam is modulated by a multi-leaf collimator (MLC), as shown in Figure~\ref{fig:rt_overview}(a). The treatment planning goal is to customize the shape and intensity of these beams for each patient's anatomy so that the tumor receives the prescribed dose while surrounding organs at risk (OARs) are spared as much as possible. Figure~\ref{fig:rt_overview}(b) shows a representative dose map for a lung case, where the high-dose region is concentrated around the tumor while dose to nearby structures is limited. Additional background on radiotherapy treatment planning  can be found in \citep{zarepisheh2022automated,romeijn2005column,shepard1999optimizing}.
		
		From an optimization perspective, treatment planning is performed by discretizing both the radiation delivery system and the patient anatomy. As shown in Figure~\ref{fig:rt_overview}(c), each beam is divided into small beamlets, and the patient anatomy is discretized into three-dimensional voxels. Let $x \in \mathbb{R}^n$ denote the vector of beamlet intensities, where $n$ is the number of beamlets, and let $D \in \mathbb{R}^{m \times n}$ denote the dose influence matrix, where $m$ is the number of voxels. The entry $D_{ij}$ represents the dose delivered to voxel $i$ by unit intensity of beamlet $j$. Therefore, the total dose distribution delivered to the patient is $d = Dx.$ For a structure $s$, such as the tumor or an OAR, we denote by $D^s$ the submatrix of $D$ corresponding to voxels in that structure.
		
		The IMRT treatment planning problem seeks a nonnegative beamlet intensity vector that provides adequate tumor irradiation and limits dose to healthy organs. A representative formulation is
		\begin{equation}
			\begin{aligned}
				\min_x \quad
				& 
				w_o \left\|D^{\mathrm{tar}}x-\mathrm{pres}\right\|_{+}^{2}
				+
				w_u \left\|D^{\mathrm{tar}}x-\mathrm{pres}\right\|_{-}^{2}
				+
				\sum_{s\in \mathcal{O}} w_s \left\|D^s x\right\|_2^2
				+
				w_{\mathrm{smooth}} \left\|Px\right\|_{2}^{2} \\
				\text{s.t.} \quad
				& \max(D^s x) \le d_s^{\max}, \qquad s\in \mathcal{S}_{\max}, \\
				& \mathrm{mean}(D^s x) \le d_s^{\mathrm{mean}}, \qquad s\in \mathcal{S}_{\mathrm{mean}}, \\
				& x \ge 0 .
			\end{aligned}
			\label{eq:radiotherapy_qp}
		\end{equation}
		Here, $D^{\mathrm{tar}}$ denotes the submatrix of $D$ associated with the tumor (target) voxels, $\mathrm{pres}$ denotes the tumor prescription dose specified by the radiation oncologist, and the one-sided penalties are defined componentwise as $\|v\|_{+}^{2}=\|\max\{v,0\}\|_2^2$ and $\|v\|_{-}^{2}=\|\max\{-v,0\}\|_2^2$. The first two terms penalize deviations of the tumor dose from prescription: the positive part penalizes overdose, while the negative part penalizes underdose. The third term penalizes dose delivered to OARs, and the smoothness term, defined through the matrix $P$, discourages large intensity differences between neighboring beamlets to improve deliverability. The constraints represent clinical requirements, including maximum-dose limits, mean-dose limits, and the physical nonnegativity of beamlet intensities. Further details on these formulations can be found in \citep{zarepisheh2022automated} and in the open-source PortPy package \citep{jhanwar2023portpy}.
		
		After standard reformulations of the piecewise-quadratic overdose and underdose penalties and the dose constraints, the model can be written in the canonical QP form \eqref{eq:org_prob}. The resulting QP instances are large because the dose influence matrix $D$ maps many beamlets, typically on the order of $10^3$--$10^4$, to many anatomical voxels, typically on the order of $10^5$--$10^6$, and because clinical protocols impose multiple structure-level constraints. The planning models underlying these instances are generated from real patient data with the open-source PortPy package developed by our team \citep{jhanwar2023portpy}. While PortPy targets radiotherapy researchers, the contribution of this work is to turn these planning problems into a benchmark for the broader optimization community: we formulate and extract 60 lung-cancer cases directly in canonical QP form, including $Q,q,G,h,A,b$, so that the instances can be used without any domain-specific knowledge of radiotherapy, and release them in the \href{https://github.com/Radiotherapy-Optimization/Warm-IP}{Warm-IP GitHub repository}.

		\begin{figure}[htbp]
			\centering
			\includegraphics[width=0.85\textwidth]{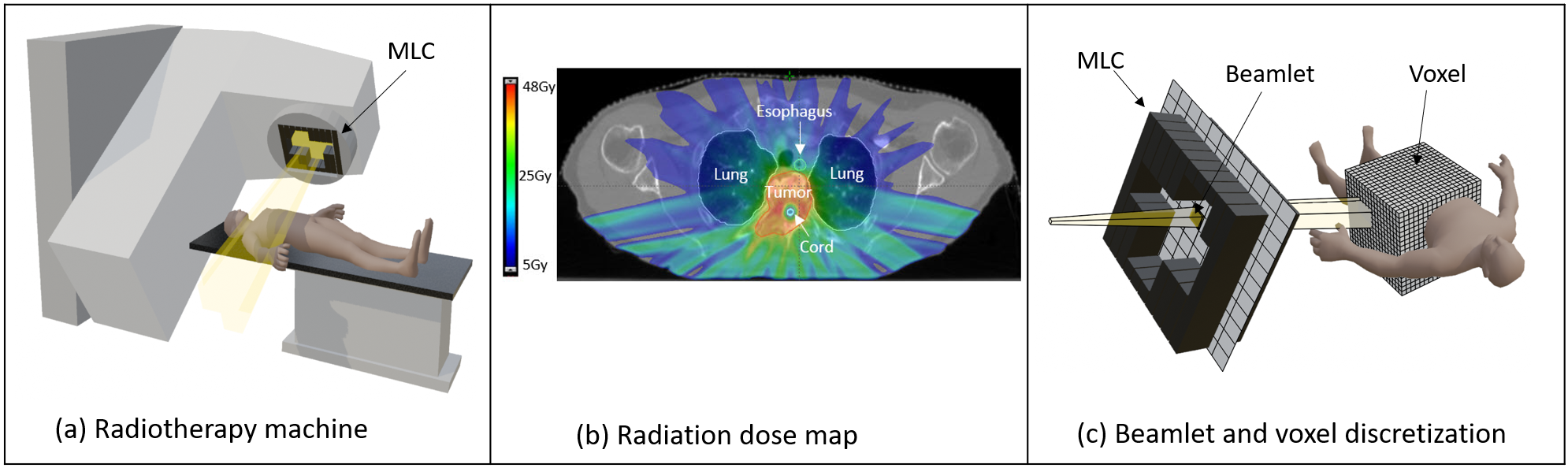}
			\caption{
				Radiotherapy treatment planning: (a) radiation beams shaped by a multi-leaf collimator (MLC); (b) a representative dose map for a lung case; (c) discretization into beamlets and voxels, linked by the dose influence matrix.
			}
			\label{fig:rt_overview}
		\end{figure}
		\subsection{Results}
		\label{subsec:results}
		
		Figure~\ref{benchmark_results} reports, separately for each benchmark class, each solver's mean runtime together with dotted companion bars indicating the percentage of instances on which it failed. In our experiments, a solver is considered to have failed on a problem instance if it returns no result or its returned solution has $\mathrm{KKT}_{\max}=\max(\text{primal residual},\,\text{dual residual},\,\text{duality gap})>10^{-3}$. The corresponding solution-quality distributions are reported in the size-aggregated analysis of Figure~\ref{overall_box}. Throughout this section, mean runtimes are computed over each solver's returned results; the full runtime distributions behind the means of Figures~\ref{benchmark_results} and~\ref{overall_box} are provided as box plots in Appendix~\ref{app:runtime_dists}. The results show clear differences in solver behavior across problem classes and problem scales.
		
		For the small-scale MPC instances, \textit{Clarabel} and \textit{PIQP} achieve the best performance. On the MM benchmark, which consists mostly of small- and medium-scale QPs with a few larger instances, \textit{Gurobi}, \textit{Clarabel}, and \textit{PIQP} achieve the best performance, combining the lowest mean runtimes ($0.5$--$0.9$ seconds) with failure rates below $2\%$. The picture reverses on the two predominantly large-scale families. On the synthetic benchmark, Warm-IP solves every instance, and its mean runtime of $44$ seconds is less than half that of any competitor (roughly $100$--$200$ seconds). On IMRT, Warm-IP is again the fastest on average ($47$ seconds) with no failures; \textit{MOSEK}'s mean is close ($55$ seconds) but it fails on $23\%$ of the instances, and the fastest fully reliable competitor, \textit{Gurobi}, is about five times slower (roughly $240$ seconds).
		
		To separate the effect of problem size from benchmark type, Figure~\ref{overall_box} aggregates all instances into two categories: small/medium instances, defined by $m+n \le 100{,}000$, and large instances, defined by $m+n>100{,}000$. This analysis confirms the trends of Figure~\ref{benchmark_results} and adds the solution-quality dimension. For small- and medium-scale problems, \textit{Clarabel} offers the best combination of speed, accuracy, and reliability (mean $0.6$ seconds, $0.5\%$ failures), with \textit{PIQP} and \textit{Gurobi} also fast, accurate, and rarely failing; Warm-IP's mean runtime is within a couple of seconds of these leaders, but it fails on $10\%$ of these instances (concentrated in the MM set). For large-scale problems, Warm-IP is the clear winner on all three axes: it is the fastest on average ($50$ seconds, versus $65$ for MOSEK, which fails on $26\%$ of the instances, and $200$ for Gurobi), fails on only $2.6\%$ of the instances (two large MM cases), and returns accurate solutions. OSQP, as a first-order method, is designed for moderate-accuracy solutions; at the tight tolerance of $10^{-6}$ used here its failure rates are high, and its results should be read in that light. These results suggest that Warm-IP is particularly effective in the regime for which it was designed: large-scale QPs where first-order methods may struggle to reach high accuracy and conventional IP methods can be computationally expensive without an effective warm start.
		
		\begin{figure}[htbp]
			\centering
			\includegraphics{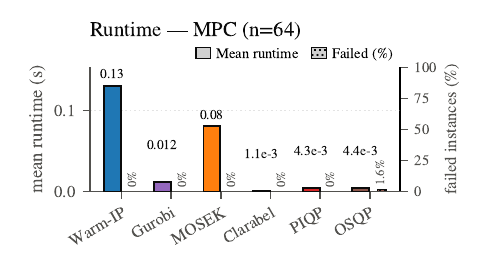}\hfill
			\includegraphics{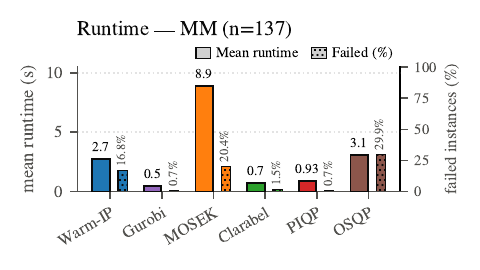}\\[2pt]
			\includegraphics{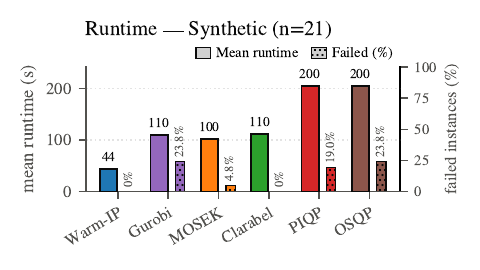}\hfill
			\includegraphics{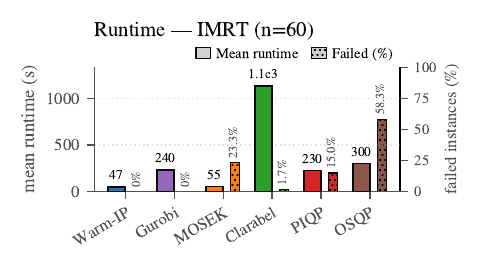}
			\caption{
				Mean runtime per benchmark class (instance counts shown in the panel titles). Solid bars show each solver's mean runtime over its returned results (value labeled); dotted bars give the percentage of failed instances, measured on the right axis. The full runtime distributions are reported in Appendix~\ref{app:runtime_dists}.
			}
			\label{benchmark_results}
		\end{figure}
		
		\begin{figure}[htbp]
			\centering
			\includegraphics{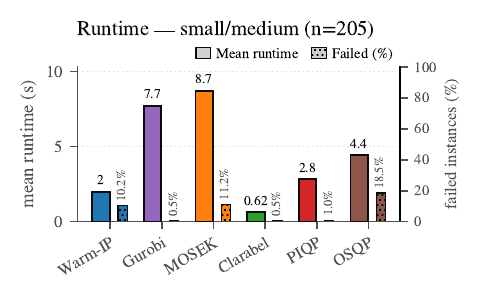}\hfill
			\includegraphics{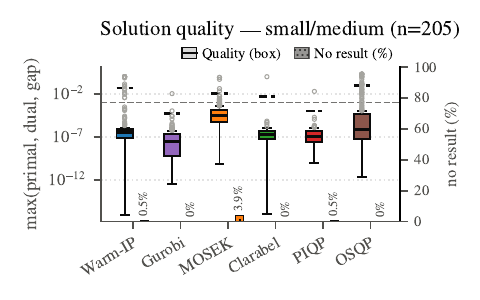}\\[2pt]
			\includegraphics{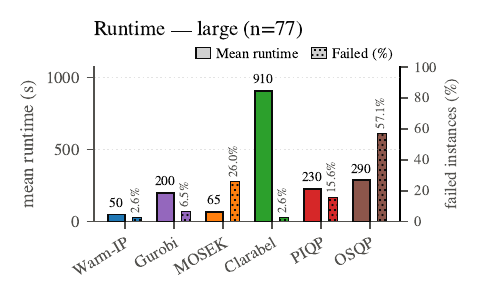}\hfill
			\includegraphics{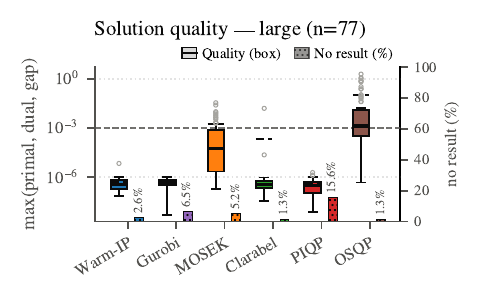}
			\caption{
				Aggregate mean runtime (left, semantics as in Figure~\ref{benchmark_results}) and solution-quality distributions (right) for small/medium (top) and large (bottom) instances (instance counts shown in the panel titles). Quality boxes cover each solver's returned results, with dotted bars giving the percentage of instances for which no result was returned; open circles mark solutions beyond the whiskers. The dashed horizontal line marks the failure threshold $\mathrm{KKT}_{\max}=10^{-3}$. The full runtime distributions are reported in Appendix~\ref{app:runtime_dists}.
			}
			\label{overall_box}
		\end{figure}
		In addition to the mean-runtime comparisons, we evaluate solvers using the \textit{performance profile} introduced by \citet{dolan2002benchmarking}, a standard tool for benchmarking optimization software. Performance profiles provide a normalized comparison across heterogeneous problem instances by measuring how close each solver is to the best-performing solver on each problem.
		
		Let $t_{p,s}$ denote the computational time required by solver $s\in\mathcal{S}$ to solve problem $p\in\mathcal{P}$. For each problem, we define the performance ratio
		\begin{equation}
			r_{p,s}
			=
			\frac{t_{p,s}}{\min_{s'\in\mathcal{S}} t_{p,s'}},
			\qquad r_{p,s}\ge 1.
			\label{eq:performance_ratio}
		\end{equation}
		Thus, $r_{p,s}=1$ indicates that solver $s$ is the fastest solver on problem $p$, while $r_{p,s}=2$ indicates that it is twice as slow as the fastest solver for that problem. If a solver fails on a problem, we assign it an infinite ratio so that the failure is reflected in the profile without removing the instance from the comparison. The performance profile of solver $s$ is defined as
		\begin{equation}
			\rho_s(\tau)
			=
			\frac{1}{|\mathcal{P}|}
			\left|
			\left\{
			p\in\mathcal{P}: r_{p,s}\le \tau
			\right\}
			\right|,
			\qquad \tau\ge 1.
			\label{eq:performance_profile}
		\end{equation}
		Therefore, $\rho_s(\tau)$ represents the fraction of problems for which solver $s$ solves the problem within a factor $\tau$ of the best runtime achieved by any solver on that problem. In particular, $\rho_s(1)$ is the fraction of problems on which solver $s$ is the fastest. Similarly, $\rho_s(1.2)$ is the fraction of problems for which solver $s$ is at most $20\%$ slower than the fastest solver, and $\rho_s(2)$ is the fraction of problems for which it is within a factor of two of the best runtime. Hence, the behavior of the curve near $\tau=1$ reflects how often a solver is the fastest or nearly fastest, while the height of the curve for larger $\tau$ reflects robustness and consistency across the benchmark set.
		
		Figure~\ref{overall_performance} shows the performance profiles for the small/medium and large problem categories. Failed solves are assigned an infinite ratio, so failures do not drop out of the comparison; instead they cap a solver's curve below one. For small- and medium-scale instances, Clarabel has the steepest rise near $\tau=1$ and the highest curve throughout, consistent with Figures~\ref{benchmark_results} and~\ref{overall_box}. OSQP rises almost as quickly: its low per-iteration cost makes it the fastest or nearly fastest on many of the instances it does solve, while its plateau reflects its failure rate. PIQP and Gurobi reach comparable fractions at larger $\tau$. Warm-IP is weakest in this regime, which is expected because its hybrid ADMM--IP design, together with the current Python implementation, introduces overhead that is less beneficial for smaller problems.
		
		In contrast, the large-scale results show the opposite trend. Warm-IP has the highest value near $\tau=1$, meaning that it is the fastest solver on the largest fraction of large instances. Its profile also rises rapidly and remains above the other solvers over most of the plotted range, indicating that Warm-IP is consistently close to the best observed runtime across the large-scale benchmark. MOSEK and Gurobi are the strongest competing solvers, but their profiles reach comparable fractions only at larger values of $\tau$, reflecting higher runtimes relative to Warm-IP. OSQP is substantially less favorable in this regime, while Clarabel and PIQP are slower but more robust.
		
		Overall, the performance profiles reinforce the conclusions from Figures~\ref{benchmark_results} and~\ref{overall_box}: Warm-IP is most advantageous for large-scale QPs, where the path-following ADMM warm start enables the IP phase to achieve high accuracy with substantially reduced computational time. The warm start also shortens the IP phase itself. As shown in Figure~\ref{fig:ip_iters}, Warm-IP requires substantially fewer IP iterations than the cold-started IP on the large-scale families (on average $17$ versus $27$ on Synthetic and $22$ versus $31$ on IMRT) which translates into mean runtimes that are less than half on Synthetic ($44$ versus roughly $100$ seconds) and about $20\%$ lower on IMRT ($47$ versus $58$ seconds). On the smaller MPC and MM families the improvement is more modest. This indicates that the ADMM phase hands the IP phase an iterate close enough to the central path to save a meaningful share of its iterations, with the payoff concentrated where each IP iteration is expensive.

		\begin{figure}[htbp]
			\centering
			\includegraphics{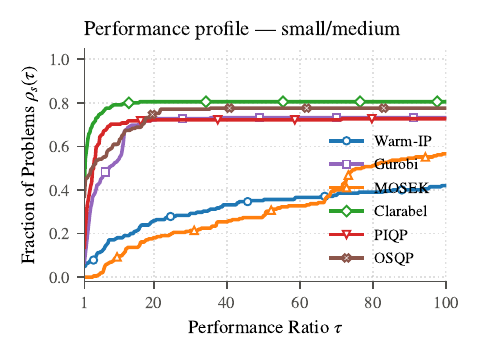}\hfill
			\includegraphics{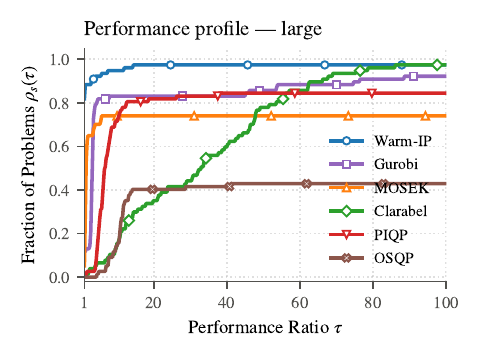}
			\caption{
				Performance profiles $\rho_s(\tau)$ (the fraction of problems solved within a factor $\tau$ of the fastest solver) on small/medium (left) and large (right) instances. Failed solves receive an infinite ratio; only solvers run on every instance in a category are included.
			}
			\label{overall_performance}
		\end{figure}
		
		\begin{figure}[htbp]
			\centering
			\includegraphics{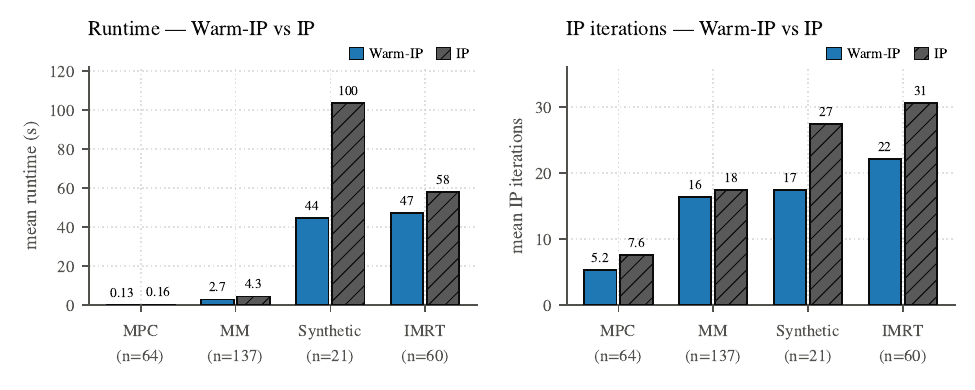}
			\caption{
				Warm-IP versus the cold-started IP per benchmark family: mean runtime (left) and mean IP iterations (right) over each method's returned results (values labeled). The full distributions are reported in Appendix~\ref{app:runtime_dists}.
			}
			\label{fig:ip_iters}
		\end{figure}
		
		\section{Conclusion and Future Research}\label{conclusion}
		In this paper, we proposed Warm-IP, a hybrid framework that combines the scalability of first-order methods with the high-accuracy convergence of primal--dual interior-point methods: a path-following variant of ADMM whose iterates (with an appropriate variable update order) satisfy the perturbed complementarity relation $SZe=\mu e$, and therefore hand the IP phase a warm start that is approximately feasible and aligned with the central path. Across a diverse benchmark suite of MM, MPC, synthetic, and real-patient radiotherapy QPs, Warm-IP substantially outperforms existing commercial and open-source solvers on large-scale instances; on the radiotherapy benchmark it solves every instance and is the fastest on average, while the strongest competitors either fail on nearly a quarter of the cases (MOSEK) or are several times slower (Gurobi).
		
		There are several directions for future research. First, the current framework is developed for convex quadratic programming. A natural and important extension is to general nonlinear programming problems, where both the objective and constraints may be nonlinear. Extending the path-following ADMM warm-start strategy to nonlinear programming would require handling nonlinear residuals, updating Hessian or quasi-Newton approximations, and potentially nonconvexity, but could substantially broaden the applicability of the method. Second, the current implementation is written primarily in Python to support rapid development and reproducibility. Future work will focus on developing a high-performance implementation in a compiled language and further exploiting sparse linear algebra, parallelism, and GPU acceleration. Finally, additional theoretical analysis is needed to characterize the convergence behavior of the combined ADMM--IP framework. In particular, it would be valuable to establish conditions under which the ADMM phase provides a provably effective central-path warm start and places the subsequent interior-point iterations within, or close to, the fast local convergence region.
		


	\section*{Acknowledgments}
This research was partially supported by NIH Cancer Center Support Grant/Core Grant (P30 CA008748).

\bibliographystyle{informs2014}
\bibliography{bibliography}

\appendix

		\section{Dual and Optimality Conditions of the Log-Barrier-Modified Problem}
	\phantomsection
	\label{sec:app:dual_derivation}
	
	We derive the Lagrange dual and the optimality conditions of the log-barrier-modified problem \eqref{eq:org_prob_modified}, whose primal form is
	\begin{equation}
		\min_{x \in \mathbb{R}^n,\; s>0} \quad \frac{1}{2}x^\top Qx + q^\top x - \mu \sum_{i=1}^{m}\log(s_i)
		\quad \text{s.t.} \quad Gx + s = h, \qquad Ax = b.
		\label{eq:app_barrier_primal}
	\end{equation}
	Let $z\in\mathbb{R}^m$ and $v\in\mathbb{R}^p$ denote the Lagrange multipliers associated with the coupling constraint $Gx+s=h$ and the equality constraint $Ax=b$, respectively. The Lagrangian is
	\begin{equation}
		\mathcal{L}(x,s,z,v) = \frac{1}{2}x^\top Qx + q^\top x - \mu\sum_{i=1}^{m} \log(s_i) + z^\top(Gx+s-h) + v^\top(Ax-b),
		\label{eq:app_barrier_lagrangian}
	\end{equation}
	and the dual function is $g(z,v)=\inf_{x,\, s>0}\mathcal{L}(x,s,z,v)$. Because the Lagrangian separates into a term involving only $x$ and a term involving only $s$, the two minimizations can be carried out independently.
	
	\emph{Minimization over $x$.} The Lagrangian is convex quadratic in $x$, so its minimizer satisfies the stationarity condition
	\begin{equation}
		Qx + G^\top z + A^\top v = -q,
		\label{eq:app_stationarity}
	\end{equation}
	which is identical to the stationarity condition of the original QP \eqref{eq:org_prob}. Substituting $\left(q+G^\top z+A^\top v\right)^\top x = -x^\top Qx$ into the $x$-part of the Lagrangian reduces it to $-\frac{1}{2}x^\top Qx$, exactly as in the standard QP dual derivation \citep{boyd2004convex}.
	
	\emph{Minimization over $s>0$.} The $s$-part of the Lagrangian is separable across components, $\sum_{i=1}^{m} \left(z_i s_i - \mu\log s_i\right)$. If $z_i \le 0$ for some $i$, this expression is unbounded below as $s_i\to\infty$, so the dual function is $-\infty$; dual feasibility therefore requires $z>0$. For $z_i>0$, setting the derivative to zero gives
	\begin{equation}
		z_i - \frac{\mu}{s_i} = 0 \qquad\Longleftrightarrow\qquad s_i = \frac{\mu}{z_i},
		\label{eq:app_s_stationarity}
	\end{equation}
	and the minimum value of the $i$-th term is $\mu - \mu\log(\mu/z_i) = \mu\log(z_i) + \mu(1-\log\mu)$.
	
	Collecting the two parts, the dual function for $z>0$ is
	\begin{equation}
		g(z,v) = -\frac{1}{2}x^\top Qx - h^\top z - b^\top v + \mu\sum_{i=1}^{m} \log(z_i) + m\mu(1-\log\mu),
		\label{eq:app_dual_function}
	\end{equation}
	with $x$ determined by the stationarity constraint \eqref{eq:app_stationarity}. Dropping the additive constant $m\mu(1-\log\mu)$, which does not affect the maximizer, yields the dual problem stated in \eqref{eq:org_prob_modified}.
	
	\emph{Optimality conditions.} Problem \eqref{eq:app_barrier_primal} is convex, and under the standing assumptions of Section~\ref{sec:methodology} strong duality holds. Its optimality conditions consist of the stationarity condition \eqref{eq:app_stationarity}; primal feasibility $Gx+s=h$ with $s>0$ and $Ax=b$; dual feasibility $z>0$; and the $s$-stationarity condition \eqref{eq:app_s_stationarity}, which in matrix form reads
	\begin{equation}
		SZe = \mu e.
		\label{eq:app_perturbed_complementarity}
	\end{equation}
	Comparing with the KKT conditions \eqref{eq:KKT_compact} of the original QP, the stationarity and primal and dual feasibility equations are unchanged, while the complementarity condition $SZe=0$ is replaced by the perturbed condition $SZe=\mu e$ (and the nonnegativity conditions $s\ge 0$, $z\ge 0$ are sharpened to strict positivity). This substantiates the claim made after \eqref{eq:org_prob_modified}.
	
	\emph{Duality gap at an arbitrary iterate.} Finally, we record the exact expression for the duality gap used in Section~\ref{subsec:adaptive_rho}. For any point $(x,s,z,v)$, let $p=\frac{1}{2}x^\top Qx+q^\top x$ and $d=-\frac{1}{2}x^\top Qx-h^\top z-b^\top v$ denote the primal and dual objective values of the original QP \eqref{eq:org_prob} evaluated at that point, as in Remark~\ref{rem:stopping_criteria}, so that
	\begin{equation}
		p - d = x^\top Qx + q^\top x + h^\top z + b^\top v.
		\label{eq:app_gap_raw}
	\end{equation}
	Substituting $q^\top x = x^\top\left(Qx+G^\top z+A^\top v+q\right) - x^\top Qx - z^\top Gx - v^\top Ax$, together with $h^\top z = z^\top(Gx+s) - z^\top(Gx+s-h)$ and $b^\top v = v^\top Ax - v^\top(Ax-b)$, gives the identity
	\begin{equation}
		p - d = s^\top z + x^\top\left(Qx+G^\top z+A^\top v+q\right) - z^\top\left(Gx+s-h\right) - v^\top\left(Ax-b\right).
		\label{eq:app_gap_identity}
	\end{equation}
	The three parenthesized vectors are precisely the stationarity and primal feasibility residuals of \eqref{eq:KKT_residual}. In particular, when these residuals vanish the duality gap equals the complementarity term $s^\top z$; along the path-following ADMM iterates, where $s^\top z = m\mu$ by Theorem~\ref{thm:admm_complementarity}, the gap therefore approaches $m\mu$ as the residuals decrease.
		

		\section{Reduction of the KKT Linear System}
		\phantomsection
		\label{sec:reduced_kkt}
		
		Each IP iteration requires solving the Newton system
		\begin{equation}
			\begin{bmatrix}
				Q & A^\top & G^\top & 0 \\
				G & 0      & 0      & I \\
				A & 0      & 0      & 0 \\
				0 & 0      & S      & Z
			\end{bmatrix}
			\begin{bmatrix}
				\Delta x \\ 
				\Delta v \\ 
				\Delta z \\ 
				\Delta s
			\end{bmatrix}
			=
			-
			\begin{bmatrix}
				r_{\mathrm d} \\
				r_{\mathrm p}^{\mathrm{ineq}} \\
				r_{\mathrm p}^{\mathrm{eq}} \\
				r_{\mathrm c}
			\end{bmatrix}.
			\label{eq:full_kkt_system}
		\end{equation}
		Although this system is sparse, it contains both the slack direction $\Delta s$ and the inequality-dual direction $\Delta z$. These variables can be eliminated to obtain a smaller linear system in only $\Delta x$ and $\Delta v$.
		
		From the second block row of \eqref{eq:full_kkt_system}, we have
		\begin{equation}
			G\Delta x + \Delta s = -r_{\mathrm p}^{\mathrm{ineq}},
		\end{equation}
		and therefore
		\begin{equation}
			\Delta s = -r_{\mathrm p}^{\mathrm{ineq}} - G\Delta x.
			\label{eq:delta_s_elimination}
		\end{equation}
		Similarly, the fourth block row gives
		\begin{equation}
			S\Delta z + Z\Delta s = -r_{\mathrm c}.
		\end{equation}
		Substituting \eqref{eq:delta_s_elimination} into this equation yields
		\begin{equation}
			S\Delta z - ZG\Delta x - Zr_{\mathrm p}^{\mathrm{ineq}} = -r_{\mathrm c}.
		\end{equation}
		Assuming $S$ is nonsingular, which holds for interior-point iterates since $s>0$, we obtain
		\begin{equation}
			\Delta z
			=
			S^{-1}ZG\Delta x
			+
			S^{-1}Zr_{\mathrm p}^{\mathrm{ineq}}
			-
			S^{-1}r_{\mathrm c}.
			\label{eq:delta_z_elimination}
		\end{equation}
		
		Substituting \eqref{eq:delta_z_elimination} into the first block row of \eqref{eq:full_kkt_system},
		\begin{equation}
			Q\Delta x + A^\top \Delta v + G^\top \Delta z = -r_{\mathrm d},
		\end{equation}
		gives
		\begin{equation}
			\left(Q + G^\top S^{-1}ZG\right)\Delta x
			+
			A^\top \Delta v
			=
			-r_{\mathrm d}
			-
			G^\top S^{-1}Zr_{\mathrm p}^{\mathrm{ineq}}
			+
			G^\top S^{-1}r_{\mathrm c}.
		\end{equation}
		The third block row remains
		\begin{equation}
			A\Delta x = -r_{\mathrm p}^{\mathrm{eq}}.
		\end{equation}
		
		Thus, defining
		\begin{equation}
			K := Q + G^\top S^{-1}ZG
			\label{eq:K_definition}
		\end{equation}
		and
		\begin{equation}
			r_{\mathrm{rhs}}
			:=
			-r_{\mathrm d}
			-
			G^\top S^{-1}Zr_{\mathrm p}^{\mathrm{ineq}}
			+
			G^\top S^{-1}r_{\mathrm c},
			\label{eq:reduced_rhs_definition}
		\end{equation}
		we obtain the reduced KKT system
		\begin{equation}
			\begin{bmatrix}
				K & A^\top \\
				A & 0
			\end{bmatrix}
			\begin{bmatrix}
				\Delta x \\
				\Delta v
			\end{bmatrix}
			=
			\begin{bmatrix}
				r_{\mathrm{rhs}} \\
				-r_{\mathrm p}^{\mathrm{eq}}
			\end{bmatrix}.
			\label{eq:reduced_kkt_system}
		\end{equation}
		
		Once $\Delta x$ and $\Delta v$ are computed from \eqref{eq:reduced_kkt_system}, the eliminated directions are recovered using
		\begin{equation}
			\Delta s = -r_{\mathrm p}^{\mathrm{ineq}} - G\Delta x
			\label{eq:recover_delta_s}
		\end{equation}
		and
		\begin{equation}
			\Delta z
			=
			S^{-1}ZG\Delta x
			+
			S^{-1}Zr_{\mathrm p}^{\mathrm{ineq}}
			-
			S^{-1}r_{\mathrm c}.
			\label{eq:recover_delta_z}
		\end{equation}
		This reduction decreases the dimension of the linear system and exposes the structured matrix $K$, which plays a central role in the computational efficiency of the IP method.


		\section{Primal--Dual Interior-Point Initialization}
		\phantomsection
		\label{sec:app_initialization}
		When IP is used without an ADMM warm start, an initial interior point must be constructed. We use a standard heuristic initialization strategy based on first relaxing the inequality constraint $Gx\le h$ and penalizing its violation through a quadratic regularization term. Specifically, we solve the equality-constrained problem
		
		\begin{equation}
			\min_{x} \quad 
			\frac{1}{2}x^\top Qx + q^\top x + \frac{1}{2}\|Gx-h\|_2^2
			\quad \text{s.t.} \quad Ax=b .
			\label{eq:pdip_init_regularized_problem}
		\end{equation}
		This problem provides an initial primal point that satisfies the equality constraints while also encouraging approximate satisfaction of the inequality constraints. Since \eqref{eq:pdip_init_regularized_problem} has only equality constraints, its solution can be obtained by solving the KKT system
		\begin{equation}
			\begin{bmatrix}
				K_0 & A^\top \\
				A & 0
			\end{bmatrix}
			\begin{bmatrix}
				x \\
				v
			\end{bmatrix}
			=
			\begin{bmatrix}
				G^\top h - q \\
				b
			\end{bmatrix},
			\label{eq:pdip_init_kkt}
		\end{equation}
		where $K_0=Q+G^\top G$ (distinct from the matrix $K$ of \eqref{eq:K_D_definition}). Solving \eqref{eq:pdip_init_kkt} gives initial estimates for the primal variable $x$ and the equality multiplier $v$.
		
		The remaining variables are the slack variable $s$ and the inequality multiplier $z$. Ignoring positivity requirements, a natural estimate for the slack variable is
		\begin{equation}
			s = h-Gx,
			\label{eq:init_s_raw}
		\end{equation}
		which is the residual of the inequality constraint. Similarly, the stationarity condition of \eqref{eq:pdip_init_kkt} can be written as
		\begin{equation}
			Qx + G^\top(Gx-h) + A^\top v + q = 0.
		\end{equation}
		Comparing this expression with the dual feasibility condition of the original primal--dual KKT system,
		\begin{equation}
			Qx + G^\top z + A^\top v + q = 0,
		\end{equation}
		suggests the initial estimate
		\begin{equation}
			z = Gx-h.
			\label{eq:init_z_raw}
		\end{equation}
		Thus, the regularized equality-constrained problem provides not only initial estimates for $x$ and $v$, but also natural residual-based estimates for $s$ and $z$.
		
		Because IP requires strictly positive slack variables and inequality multipliers, the raw estimates in \eqref{eq:init_s_raw}--\eqref{eq:init_z_raw} are shifted if necessary. Specifically, we set
		\begin{equation}
			s^{(0)}
			=
			s
			+
			\max\{0,\,1-\min_i s_i\}\,e,
			\label{eq:initial_slack_shifted}
		\end{equation}
		and
		\begin{equation}
			z^{(0)}
			=
			z
			+
			\max\{0,\,1-\min_i z_i\}\,e,
			\label{eq:initial_dual_shifted}
		\end{equation}
		where $e$ denotes the vector of ones. This uniform shift ensures that all components of $s^{(0)}$ and $z^{(0)}$ are at least one, and therefore strictly positive. The resulting point
		\[
		(x^{(0)},v^{(0)},s^{(0)},z^{(0)})
		\]
		is not guaranteed to satisfy all KKT conditions, but it provides a numerically stable interior starting point for the IP iterations.

		\section{Step Size Selection}
		\phantomsection
		\label{app:step_size}
		
		To keep the iterates strictly inside the positive orthant, we choose the step size using a fraction-to-the-boundary rule. Given the Newton directions $\Delta s$ and $\Delta z$, the primal and dual step sizes $\alpha_s,\alpha_z \in (0,1]$ are selected so that
		\[
		s^{(k)}+\alpha_s \Delta s \ge (1-\tau)s^{(k)}, 
		\qquad
		z^{(k)}+\alpha_z \Delta z \ge (1-\tau)z^{(k)},
		\]
		where the inequalities are interpreted componentwise and $\tau \in (0,1)$ is chosen close to one. In our implementation, we use $\tau=0.995$. These conditions lead to the closed-form choices
		\begin{equation}
			\alpha_s
			=
			\min\left\{
			1,\;
			\min_{i:\Delta s_i<0}\frac{\tau s_i^{(k)}}{-\Delta s_i}
			\right\},
			\qquad
			\alpha_z
			=
			\min\left\{
			1,\;
			\min_{i:\Delta z_i<0}\frac{\tau z_i^{(k)}}{-\Delta z_i}
			\right\}.
			\label{eq:fraction_to_boundary_alpha_pd}
		\end{equation}
		If all components of $\Delta s$ or $\Delta z$ are nonnegative, the corresponding minimum is omitted. The final step size is then taken as $\alpha=\min\{\alpha_s,\alpha_z\}.$ This choice guarantees $s^{(k+1)}>0$ and $z^{(k+1)}>0$ while allowing the algorithm to take nearly full Newton steps whenever possible.
		
		\section{Mehrotra Predictor--Corrector Step}
		\phantomsection
		\label{app:mehrotra}
		
		A basic path-following IP method reduces the barrier parameter $\mu$ gradually, for example by setting the target complementarity to $\sigma\mu$ with a fixed $\sigma \in (0,1)$. Mehrotra's predictor--corrector technique provides a more adaptive and effective strategy for choosing this reduction at each iteration \citep{mehrotra1992implementation}. The method consists of two main steps. First, in the \textit{predictor} step, an affine-scaling direction is computed by solving the Newton system \eqref{eq:kkt_system} with the centering term removed, i.e., with target complementarity $\mu=0$. Let
		\[
		(\Delta x^{\mathrm{aff}},\Delta v^{\mathrm{aff}},
		\Delta z^{\mathrm{aff}},\Delta s^{\mathrm{aff}})
		\]
		denote this affine direction. This direction predicts how much progress the algorithm can make toward reducing complementarity if it moves aggressively toward the boundary. Using the largest admissible affine step lengths, $\alpha_s^{\mathrm{aff}}$ and $\alpha_z^{\mathrm{aff}}$, the predicted complementarity is estimated as
		\begin{equation}
			\mu_{\mathrm{aff}}
			=
			\frac{
				\left(s+\alpha_s^{\mathrm{aff}}\Delta s^{\mathrm{aff}}\right)^\top
				\left(z+\alpha_z^{\mathrm{aff}}\Delta z^{\mathrm{aff}}\right)
			}{m},
			\qquad
			\mu=\frac{s^\top z}{m}.
			\label{eq:mu_aff}
		\end{equation}
		The centering parameter is then chosen adaptively as
		\begin{equation}
			\sigma
			=
			\left(\frac{\mu_{\mathrm{aff}}}{\mu}\right)^3,
			\label{eq:mehrotra_sigma}
		\end{equation}
		which is typically clipped to the interval $[0,1]$ in implementations. Thus, if the affine step is predicted to substantially reduce complementarity, then $\sigma$ becomes small and the method emphasizes progress toward optimality. Conversely, if the affine step is less effective, then $\sigma$ remains larger and the method places more emphasis on centering.
		
		Second, in the \textit{corrector} step, the Newton system is solved again with a modified complementarity residual:
		\begin{equation}
			\label{eq:corr_compl_res}
			r_{\mathrm c}^{\mathrm{pc}}=\underbrace{SZe}_{\text {prediction }} +\underbrace{\Delta S^{\mathrm{aff}}\Delta Z^{\mathrm{aff}}e}_{\text {correction }}-\underbrace{\sigma \mu e}_{\text {centering }}    
		\end{equation}
		where $\Delta S^{\mathrm{aff}}=\operatorname{diag}(\Delta s^{\mathrm{aff}})$ and
		$\Delta Z^{\mathrm{aff}}=\operatorname{diag}(\Delta z^{\mathrm{aff}})$. The term $-\sigma\mu e$ pulls the iterate back toward the central path, while the second-order correction term 
		$\Delta S^{\mathrm{aff}}\Delta Z^{\mathrm{aff}}e$ compensates for the nonlinear change in the complementarity product $SZe$ along the affine direction. Intuitively, the predictor step estimates how far the method can move toward optimality, and the corrector step adjusts this direction to maintain a favorable balance between reducing complementarity and staying sufficiently close to the central path. This predictor--corrector strategy typically permits longer and more effective steps than a purely path-following method, thereby reducing the number of IP iterations required in practice.
		
		\section{Primal--Dual Interior-Point Pseudocode}
		\phantomsection
		\label{app:pdip_alg}
		
		Algorithm~\ref{alg:pdip} assembles the components of Section~\ref{sec:pdip}---the Newton system \eqref{eq:kkt_system}, the initialization of Appendix~\ref{sec:app_initialization}, the fraction-to-the-boundary rule of Appendix~\ref{app:step_size}, and the predictor--corrector strategy of Appendix~\ref{app:mehrotra}---into the complete primal--dual interior-point method.
		
		\algrenewcommand\algorithmiccomment[1]{\hfill{\footnotesize$\triangleright$ #1}}
		
		\algblockdefx[Mehrotra]{MehrotraPredictor}{EndMehrotraPredictor}
		[1]{\textbf{Mehrotra predictor--corrector:} #1}
		{\textbf{end Mehrotra predictor--corrector}}
		
		\begin{algorithm}[htbp]
			\small
			\caption{Primal--Dual Interior-Point Method for Convex QP}
			\label{alg:pdip}
			\begin{algorithmic}[1]
				\Require QP data $(Q,q,G,h,A,b)$, tolerance $\varepsilon>0$, fraction-to-boundary parameter $\tau\in(0,1)$
				\Ensure Approximate primal--dual solution $(x,v,z,s)$
				
				\State Initialize $(x^{(0)},v^{(0)},z^{(0)},s^{(0)})$ using Appendix~\ref{sec:app_initialization}, and set $k \gets 0$
				\Statex {\footnotesize$\triangleright$ cache the symbolic factorization for reuse in later KKT solves; see Remark~\ref{rem:sparsity_reuse}.}

				
				\While{stopping criteria in Remark~\ref{rem:stopping_criteria} are not satisfied}
				\MehrotraPredictor{}
				\State Set $\mu \gets 0$ and compute the KKT residual 
				$F(x^{(k)},v^{(k)},z^{(k)},s^{(k)})$ using \eqref{eq:perturbed_kkt_residual}
				\State Solve KKT system \eqref{eq:kkt_system} for affine step $(\Delta x^{\text{aff}},\Delta v^{\text{aff}},\Delta z^{\text{aff}},\Delta s^{\text{aff}})$
				\Statex \hspace{\algorithmicindent}{\footnotesize$\triangleright$ (1) better to solve the reduced KKT system \eqref{eq:reduced_kkt_remark}; see Remark~\ref{rem:reduced_kkt}, (2) reuse the symbolic factorization.}  
				\State Compute $\mu_{\mathrm{aff}}$ and $\sigma$ using \eqref{eq:mu_aff}--\eqref{eq:mehrotra_sigma}
				\State Compute the corrected complementarity residual using \eqref{eq:corr_compl_res}
				\EndMehrotraPredictor

				\State Resolve the KKT system \eqref{eq:kkt_system} for corrected step $(\Delta x,\Delta v,\Delta z,\Delta s)$
				\Statex \hspace{\algorithmicindent} {\footnotesize$\triangleright$ solve the reduced KKT system and reuse the prior factorization.}
				\State Compute the step size $\alpha$ using the fraction-to-the-boundary rule in \eqref{eq:fraction_to_boundary_alpha_pd}
				\State Update $(x,v,z,s)$ using \eqref{eq:pdip_update}
				\State $k \gets k+1$
				
				\EndWhile
				
				\State \Return $(x^{(k)},v^{(k)},z^{(k)},s^{(k)})$
				
			\end{algorithmic}
		\end{algorithm}
		
		\section{\texorpdfstring{Adaptive $\rho$ Update Strategy}{Adaptive rho Update Strategy}}
		\phantomsection
		\label{app:adaptive_rho}
		
		The ADMM penalty parameter $\rho$ trades off primal against dual progress: a large $\rho$ enforces the coupling constraint $Gx+s=h$ aggressively and drives down the primal residual, while a small $\rho$ favors the dual residual. Following the residual-balancing strategy of \citet{he2000alternating}, as used in OSQP \citep{stellato2020osqp}, we rescale $\rho$ by
		\begin{equation}
			\rho_{\mathrm{scale}}
			=
			\left(
			\frac{r_{\mathrm p}}{r_{\mathrm d}+\delta}
			\right)^{1/2},
			\label{eq:rho_scale}
		\end{equation}
		where $r_{\mathrm p}$ and $r_{\mathrm d}$ are the ADMM primal and dual residual norms, $\delta>0$ avoids division by zero, and the square-root exponent dampens the update. Because changing $\rho$ modifies the $x$-update matrix in \eqref{eq:admm_xv_update} and thus requires a new numerical factorization, $\rho$ is updated only when the imbalance is large: for a threshold $\tau_\rho>1$ (we use $\tau_\rho=2$),
		\begin{equation}
			\rho \leftarrow \rho\,\rho_{\mathrm{scale}}
			\quad
			\text{only if}
			\quad
			\rho_{\mathrm{scale}}>\tau_\rho
			\quad \text{or} \quad
			\rho_{\mathrm{scale}}<\frac{1}{\tau_\rho}.
			\label{eq:rho_update}
		\end{equation}
		Although Theorem~\ref{thm:admm_complementarity} is stated for fixed $\rho$, its proof only requires the $s$- and $z$-updates within one iteration to share the same $\rho$; since the rule modifies $\rho$ only at the end of an iteration, the relation $SZe=\mu e$ is preserved.
		
		\section{Runtime and IP-Iteration Distributions}
		\phantomsection
		\label{app:runtime_dists}
		
		Figures~\ref{fig:box_family}--\ref{fig:box_vsip} report the full distributions summarized by the mean bars of Figures~\ref{benchmark_results}, \ref{overall_box}, and~\ref{fig:ip_iters}. Boxes cover each solver's returned results (solutions exceeding the failure tolerance included); whiskers extend to the most extreme values within $1.5\times$ the interquartile range of the box, and open circles mark solves beyond the whiskers. In Figures~\ref{fig:box_family} and~\ref{fig:box_pooled}, dotted bars give the percentage of failed instances (no result returned, or $\mathrm{KKT}_{\max}>10^{-3}$), measured on the right axis. For readability, the runtime axes are capped near the 95th percentile of the plotted values; a few extreme outliers lie beyond the axis.
		
		\begin{figure}[htbp]
			\centering
			\includegraphics{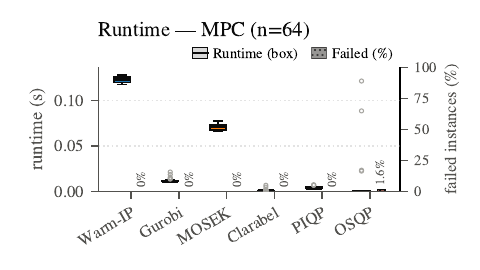}\hfill
			\includegraphics{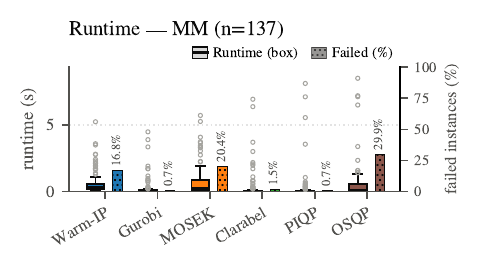}\\[2pt]
			\includegraphics{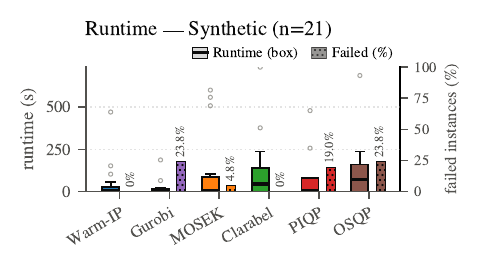}\hfill
			\includegraphics{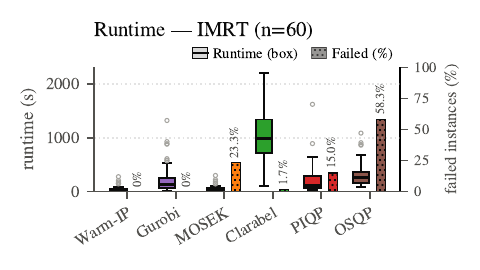}
			\caption{
				Runtime distributions per benchmark class (instance counts shown in the panel titles), summarized by the mean-runtime bars of Figure~\ref{benchmark_results}.
			}
			\label{fig:box_family}
		\end{figure}
		
		\begin{figure}[htbp]
			\centering
			\includegraphics{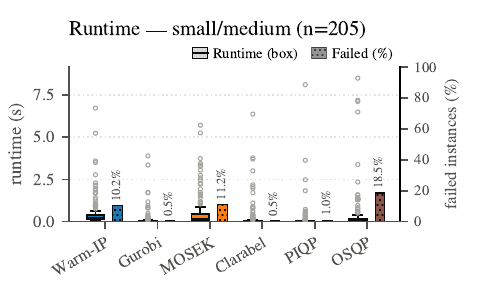}\hfill
			\includegraphics{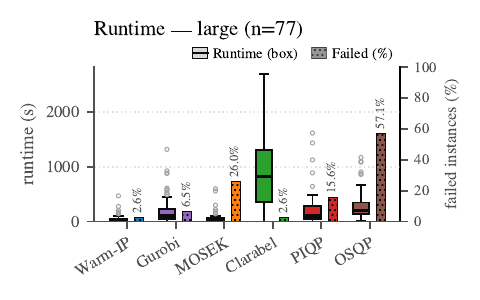}
			\caption{
				Aggregate runtime distributions for small/medium (left) and large (right) instances, summarized by the mean-runtime bars of Figure~\ref{overall_box}.
			}
			\label{fig:box_pooled}
		\end{figure}
		
		\begin{figure}[htbp]
			\centering
			\includegraphics{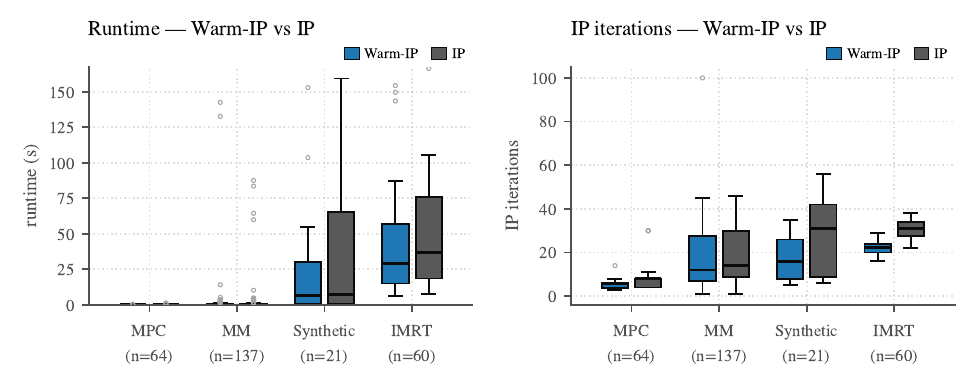}
			\caption{
				Warm-IP versus the cold-started IP per benchmark family: runtime (left) and IP-iteration (right) distributions, summarized by the mean bars of Figure~\ref{fig:ip_iters}.
			}
			\label{fig:box_vsip}
		\end{figure}

\end{document}